\documentclass{amsart}
\usepackage{amssymb}
\usepackage[v2,cmtip]{xy}

\DeclareMathAlphabet{\scr}{U}{rsfs}{m}{n}

\def\draftdate{December 31, 2010}

\newcommand{\LC}[1][n]{\oC_{#1}}
\newcommand{\LCR}[1][n]{\oC_{#1}\,R}
\newcommand{\ld}{\mathbf{L}}
\newcommand{\TQ}{Q^{\ld}_{\LC}}

\newcommand{\alg}[2]{#1\,#2}
\newcommand{\gO}{\oC}
\newcommand{\bgsO}{\bC}
\newcommand{\bgO}{\bgsO^{\sharp}}
\newcommand{\ssdot}{\bullet}
\newcommand{\subdot}{_{\ssdot}}

\newcommand{\MA}[1][{}]{{MA}_{#1}}
\newcommand{\otos}[1][{s}]{\underline{#1}}
\newcommand{\otoj}{\otos[j]}
\newcommand{\otom}{\otos[m]}
\newcommand{\oton}{\otos[n]}
\newcommand{\otop}{\otos[p]}
\mathchardef\varDelta="7101
\newcommand{\DDelta}{{\mathbf \varDelta}}
\newcommand{\oDDelta}{\underline\DDelta}
\newcommand{\ma}{\#a}

\newcommand{\MoAlg}[1]{M({#1})}
\newcommand{\embf}{\rho_{\mathrm{first}}}
\newcommand{\embl}{\rho_{\mathrm{last}}}
\newcommand{\bLC}[1][{n}]{\bC^{\sharp}_{#1}}
\newcommand{\btLC}[1][{n}]{\bC_{#1}}

\newcommand{\NC}[1][n]{\widetilde\oC_{#1}}
\newcommand{\btNC}[1][n]{\widetilde\bC_{#1}}
\newcommand{\bNC}[1][n]{\widetilde\bC^{\sharp}_{#1}}

\newcommand{\MN}[1][{}]{{MN}_{#1}}
\newcommand{\tB}{\tilde B}
\newcommand{\tE}{\tilde E}
\newcommand{\tos}[1][{s}]{\mathbf{#1}}
\newcommand{\toj}{\tos[j]}
\newcommand{\tom}{\tos[m]}

\def\n^#1_#2{\nu_{#2}}

\newcommand{\hotimes}{\mathbin{\widehat\otimes}}

\newcommand{\TQN}{Q^{\ld}}
\newcommand{\SSB}{\oB}

\DeclareMathOperator{\sd}{\mathrm{sd}_{2}}

\newcommand{\putatop}[2]{\genfrac{}{}{0pt}{}{#1}{#2}}

\newcommand{\subseg}[2]{%
\overbrace{\hbox to #2{\hss%
\vbox to 8pt{\vss\hrule height2pt width2pt depth1pt}%
\vbox to 8pt{\vss\hrule height1pt width#2 depth0pt}%
\vbox to 8pt{\vss\hrule height2pt width2pt depth1pt}%
\vbox to 0pt{\vss\hrule height0pt width0pt depth8pt}%
\hss}}^{#1}}

\newcommand{\bC}{{\mathbb{C}}}

\newcommand{\bR}{{\mathbb{R}}}

\let\catsymbfont\mathfrak 

\newcommand{\aD}{{\catsymbfont{D}}}
\newcommand{\aM}{{\catsymbfont{M}}}

\let\opsymbfont\mathcal 

\newcommand{\oA}{{\opsymbfont{A}}}
\newcommand{\oB}{{\opsymbfont{B}}}
\newcommand{\oC}{{\opsymbfont{C}}}
\newcommand{\oI}{{\opsymbfont{I}}}

\newcommand{\oP}{{\opsymbfont{P}}}
\newcommand{\oT}{{\opsymbfont{T}}}

\newcommand{\iso}{\cong}     
\newcommand{\sma}{\wedge}    

\renewcommand{\to}{\mathchoice{\longrightarrow}{\rightarrow}{\rightarrow}{\rightarrow}}
\newcommand{\from}{\mathchoice{\longleftarrow}{\leftarrow}{\leftarrow}{\leftarrow}}

\def\quickop#1{\expandafter\DeclareMathOperator\csname #1\endcsname{#1}}
\quickop{Id}\quickop{id}

\newtheorem{thm}[equation]{Theorem}

\newtheorem{lem}[equation]{Lemma}
\newtheorem{prop}[equation]{Proposition}

\theoremstyle{definition}
\newtheorem{defn}[equation]{Definition}
\newtheorem{cons}[equation]{Construction}
\newtheorem{hyp}[equation]{Hypothesis}

\theoremstyle{remark}
\newtheorem{rem}[equation]{Remark}
\newtheorem{example}[equation]{Example}

\numberwithin{equation}{section}

\newcommand{\term}[1]{\emph{#1}}

\begin{document}

\title{Homology of $E_{n}$ Ring Spectra and Iterated $THH$}
\author{Maria Basterra}
\address{Department of Mathematics, University of New Hampshire, Durham, NH}
\email{basterra@math.unh.edu}
\author{Michael A. Mandell}
\address{Department of Mathematics, Indiana University, Bloomington, IN}
\email{mmandell@indiana.edu}
\thanks{The second author was supported in part by NSF grant DMS-0804272}

\date{\draftdate}


\begin{abstract}
We describe an iterable construction of $THH$ for an $E_{n}$ ring
spectrum.  The reduced version is an iterable bar construction and its
$n$-th iterate gives a model for the shifted cotangent complex at the
augmentation, representing reduced topological Quillen homology of an
augmented $E_{n}$ algebra.
\end{abstract}

\maketitle

\section{Introduction}

Over the past two decades, topological Hochschild homology ($THH$) and
its refinement topological cyclic homology ($TC$) have become standard
tools in algebraic topology and algebraic $K$-theory.  Waldhausen
originally conjectured the theory $THH$ and used an ad hoc version of
it to split the algebraic $K$-theory of spaces into stable homotopy
theory and stable pseudo-isotopy theory \cite{WaldII}.  The theory
$TC$ provides the key tool in the proof of the $K$-theoretic Novikov
conjecture by B\"okstedt, Hsiang, and Madsen \cite{BHM}, and in the
algebraic $K$-theory computations pioneered by Hesselholt and
Madsen \cite{HM1,HM2,HM3}.  $K$-theory computations in $TC$ also led
to Ausoni and Rognes' mysterious chromatic red shift phenomenon
\cite{AusoniRognes}.  Partly because of this, recently interest has
grown in iterating $THH$ and $TC$ and in higher versions of $THH$ and
$TC$ studied by Pirashvili \cite{Pirashvili} and by Brun, Carlsson,
Douglas, and Dundas \cite{HigherTC1,HigherTC2}.

Although $THH$ makes sense for any ring or $A_{\infty}$ ring spectrum
$A$, the applications to $K$-theory computations typically take
advantage of the multiplicative structure on $THH(A)$ and in
\cite{AusoniRognes} the extended power operation structure on $THH(A)$
that arises only when $A$ has more structure, for example that of an
$E_{\infty}$ ring spectrum.  Likewise, to iterate $THH$, $THH(A)$ must
have at least an $A_{\infty}$ multiplication. In \cite{MultTHH}, Brun,
Fiedorowicz, and Vogt showed that when $A$ is an $E_{n}$ ring spectrum
$THH(A)$ is an $E_{n-1}$ ring spectrum; in this case $THH$ can be
iterated up to $n$ times, and $THH(A)$ admits certain power
operations.  The construction in \cite{MultTHH} requires replacing $A$ by an
equivalent ring spectrum over a different operad.  As a consequence,
to iterate the construction requires re-approximation at each stage.

This paper describes a construction of $THH$ of $E_{n}$ ring spectra
which is iterable without re-approximation.  Working in one of the
modern categories of spectra, such as the category of EKMM $S$-modules
\cite{ekmm} or the category of symmetric spectra of topological spaces
\cite{hss,MMSS}, we study algebras over the little $n$-cubes operad
$\LC$ of Boardman and Vogt \cite{BV}. In fact, because the output of
our $THH$ construction is not quite a $\LC[n-1]$ algebra, we work with
a mild generalization called a ``partial $\LC$-algebra'', which we
review in Section~\ref{secpart}.  A partial $\LC$-algebra is a partial
$\LC[1]$-algebra by neglect of structure; we construct a cyclic bar
construction version of
$THH$ for partial $\LC[1]$-algebras and prove the following theorem.

\begin{thm}\label{mainthh}
For a partial $\LC$-algebra $A$ satisfying mild technical
hypotheses, 
the cyclic bar construction $THH(A)$ 
is naturally a partial $\LC[n-1]$-algebra. 
\end{thm}

The mild technical hypotheses amount to the partial algebra
generalization of the usual hypothesis on algebras that the inclusion
of the unit is a cofibration of the underlying $S$-modules (or
symmetric or orthogonal spectra).  We write out the hypotheses explicitly
as Hypothesis~\ref{hypprop} below and show in
Proposition~\ref{propitprop} that it is inherited on $THH(A)$,
allowing iteration.

When working in the context of symmetric spectra, there are two
different constructions of $THH$ of an $S$-algebra: A cyclic bar
construction as in the previous theorem and the original construction
of B\"okstedt.  Because only
B\"okstedt's construction is known to be suitable for constructing
$TC$, the previous theorem still leaves open the problem of directly
constructing an iterable version of $TC$ from an $E_{n}$ ring structure;
the authors plan to return to this question in a future paper.  On the
other hand, the cyclic bar construction does
admit a relative variant for partial $\LCR$-algebras (see
Definition~\ref{defpart}), where we use a different base commutative
$S$-algebra $R$ in place of the sphere spectrum $S$.

\begin{thm}\label{mainrel}
For $R$ a commutative
$S$-algebra, and $A$ a partial $\LCR$-algebra satisfying
Hypothesis~\ref{hypprop}, $THH^{R}(A)$ is a 
partial $\LCR[n-1]$-algebra.
\end{thm}

The relative construction also admits a reduced version for augmented
algebras, which amounts to taking coefficients in $R$.  This is the
analogue of the bar construction of an augmented algebra, so we write
$B^{R}(A)$ or $BA$ rather than $THH^{R}(A;R)$.  For a partial augmented
$\LCR$-algebra $A$, $BA$ is a partial augmented $\LCR[n-1]$-algebra,
and so we can iterate the bar construction up to $n$ times.  Up to a
shift, the fiber of the augmentation $B^{n}A\to R$ is the
$R$-module of $\LCR$-algebra derived 
indecomposables representing reduced topological Quillen homology (see
Section~\ref{sectqh} for a review of this theory).

\begin{thm}\label{maintq}
For an augmented $\LCR$-algebra $A$, there is a natural
isomorphism in the derived category of $R$-modules 
\[
R\vee \Sigma^{n}\TQ(A)\iso B^{n}A,
\]
where $\TQ$ denotes the $\LCR$-algebra cotangent complex at the
augmentation (the derived $\LCR$-algebra indecomposables).
\end{thm}

We regard the previous result as the main theorem in this paper: It
allows an inductive approach to obstruction theory 
for connective $E_{n}$ ring spectra.  We explain this idea and apply it
in a future paper to prove that $BP$ is an $E_{4}$ ring spectrum.
Other non-iterative versions of Theorem~\ref{maintq} can be found
in~\cite{FrancisThesis} and~\cite{DAGVI}. An algebraic version can be
found in \cite{FresseItBar}.

\subsection*{Terminology}

Because partial $\gO$-algebras (for various $\gO$) play a fundamental
role in the structure and results of the paper, we will sometimes use the
terminology ``true $\gO$-algebra'' (for $\gO$-algebras in the usual
sense) when necessary for emphasis, for contrast, or to avoid confusion
with the terminology ``partial $\gO$-algebra''.

\subsection*{Outline}

Section~\ref{secpart} reviews the definition of a partial algebra over
an operad and the Kriz-May rectification theorem, which gives an
equivalence between the homotopy category of partial algebras and the
homotopy category of true algebras over an operad.
Sections~\ref{secmoore} and~\ref{secncyc} construct the cyclic
nerve of a partial $\LCR[1]$-algebra, breaking the construction into
two steps.  For a $\LCR[1]$-algebra $A$, we construct in Section~\ref{secmoore}
a closely related partial associative $R$-algebra, called
the ``Moore'' partial algebra, which has the same relationship to
$A$ as the Moore loop space has to the loop space.  In
Section~\ref{secncyc}, we construct the cyclic nerve of a partial
associative $R$-algebra and study its multiplicative structure when
applied to  the Moore algebra of a $\LCR[n]$-algebra, proving
Theorems~\ref{mainthh} and ~\ref{mainrel} above.  We produce the
iterated bar construction for augmented partial 
$\LCR$-algebras in Section~\ref{secbar}
and for non-unital $\LCR$-algebras in Section~\ref{secbar2}.
Section~\ref{sectqh} proves Theorem~\ref{maintq}.  Finally,
Section~\ref{secopn} proves two compatibility results for the
$E_{n-1}$-structure on the bar construction: it shows that the usual
diagonal map on the bar 
construction preserves the multiplication (Theorem~\ref{thmhopf}) and
that the bar construction has the expected behavior with respect to
power operations (Theorem~\ref{thmopn}).

\section{Partial Operadic Algebras}
\label{secpart}

In this section, we give a brief review of the definition and basic
homotopy theory of partial algebras over an operad.  We review the
Kriz-May rectification theorem, which shows that any partial
algebra over an appropriate operad is naturally weakly
equivalent to a true algebra over that same operad. This in
particular shows how to recover an operadic algebra from a partial
operadic algebra, gives an equivalence of the homotopy theory of
partial operadic algebras and of true operadic algebras, and justifies
the perspective of the statements of the main theorems in Section~1.

In this section and throughout this paper, we work in one of the
modern categories of spectra of \cite{MMSS} or \cite{ekmm}.  We write
$\aM_{S}$ for any of these categories, and refer to an object in it as
an ``$S$-module''; we write $\sma_{S}$ for the smash product,
reserving $\sma$ for the smash product with a based space.  For a
commutative $S$-algebra $R$, we have the category of $R$-modules
$\aM_{R}$ that has a symmetric monoidal product $\sma_{R}$.  We avoid
needless redundancy by typically working in $\aM_{R}$; the case of
$S$-modules being precisely the special case $R=S$.  

In our cases of interest, we work with the little $n$-cubes operads
$\LC$.  With this in mind, we fix an operad $\gO$ in spaces such that
each $\gO(m)$ is a free $\Sigma_{m}$-CW complex.  Then in
our context, a $\gO$-algebra (or 
$\alg{\gO}{R}$-algebra when $R$ needs to be made explicit) consists of
an $R$-module $A$ and maps 
\[
\gO(m)_{+}\sma_{\Sigma_{m}} (A\sma_{R}\dotsb \sma_{R}A)\to A
\]
satisfying certain associativity and unit properties.  A partial
$\gO$-algebra replaces the smash powers with an equivalent system of
$R$-modules.

\begin{defn}\label{defpps}
An \term{op-lax power system} of $R$-modules consists of 
\begin{enumerate}
\item A sequence of $R$-modules $X_{1}, X_{2},\dotsc$,
\item A $\Sigma_{m}$ action on $X_{m}$, and
\item A $\Sigma_{m}\times \Sigma_{n}$-equivariant map
$\lambda_{m,n}\colon X_{m+n}\to X_{m}\sma_{R}X_{n}$ for each $m,n$
\end{enumerate}
such that the following diagrams commute for all $m,n,p$, where $\tau_{m,n}\in \Sigma_{m+n}$ denotes the permutation that switches
the first block of $m$ past the last block of $n$.
\[
\xymatrix{%
X_{m+n}\ar[r]^-{\lambda_{m,n}}\ar[d]_{\tau_{m,n}}
&X_{m}\sma_{R}X_{n}\ar[d]^{\tau}&
X_{m+n+p}\ar[r]^-{\lambda_{m,n+p}}\ar[d]_{\lambda_{m+n,p}}
&X_{m}\sma_{R}X_{n+p}\ar[d]^{\lambda_{n,p}}\\
X_{n+m}\ar[r]_-{\lambda_{n,m}}&
X_{n}\sma_{R}X_{m}&
X_{m+n}\sma_{R}X_{p}\ar[r]_-{\lambda_{m,n}}
&X_{m}\sma_{R}X_{n}\sma_{R}X_{p}
}
\]
We write $\lambda_{m,n,p}$ for the common composite in the righthand
diagram, and more 
generally, $\lambda_{m_{1},\dotsc,m_{r}}$ for the iterated composites 
\[
X_{m_{1}+\dotsb+m_{r}}\to X_{m_{1}}\sma_{R}\dotsb \sma_{R}X_{m_{r}}.
\]
We write $X_{0}=R$ and take $\lambda_{0,n}$ and $\lambda_{n,0}$ to be
the (inverse) unit isomorphisms.  A map of op-lax power systems $A\to
B$ consists of equivariant maps $X_{n}\to Y_{n}$, which make the
evident diagrams in the structure maps $\lambda_{m,n}$ commute. An
op-lax power system is a 
\term{partial power system} when the maps
$\lambda_{m_{1},\dotsc,m_{r}}$ are weak equivalences for all
$m_{1},\dotsc,m_{r}$.
\end{defn}

\begin{example}
For an $R$-module $X$, we get a partial power system $X_{m}=X^{(m)}$
($m$-th smash power over $R$ 
for some fixed association) with the usual symmetric group actions and
the maps $\lambda_{m,n}$ the associativity isomorphism.  We call such
a partial power system a \term{true power system}.
\end{example}

This definition depends strongly on the underlying symmetric
monoidal category $\aM_{R}$: An op-lax power system in $R$-modules
does not have an underlying op-lax power system in $S$-modules.
Definition~\ref{defpps} provides the appropriate framework for the
cyclic bar construction of $THH$.  (A version suitable for the
B\"okstedt construction of $THH$ requires an ``external'' smash
product formulation that is significantly more complex.)

We define a \term{weak equivalence} of partial power systems as a map
$X\to Y$ that is a weak equivalence $X_{1}\to Y_{1}$.  This
compensates for the awkward fact that the smash product of $R$-modules
does not preserve all weak equivalences.  To help alleviate this
difficulty, we introduce the following additional terminology. 

\begin{defn}\label{deftidy}
We say that a partial power system $X$ is \term{tidy} when the canonical
maps
\[
X_{1}\sma_{R}^{\ld}\dotsb \sma_{R}^{\ld}X_{1}\to 
X_{1}\sma_{R}\dotsb\sma_{R}X_{1}
\]
from the derived smash powers to the point-set smash powers of $X_{1}$
are isomorphisms in the derived category  $\aD_{R}$.
\end{defn}

We note that a partial power system $X$ is tidy if and only
if the natural maps
\[
X_{j_{1}}\sma_{R}^{\ld}\dotsb \sma_{R}^{\ld}X_{j_{r}}\to 
X_{j_{1}}\sma_{R}\dotsb\sma_{R}X_{j_{r}}
\]
are all isomorphisms in $\aD_{R}$.  To see this, consider the
following commutative diagram in $\aD_{R}$, where 
$j=\sum j_{i}$.
\[
\xymatrix@R-1pc@C-2pc{%
X_{j_{1}}\sma_{R}^{\ld}\dotsb \sma_{R}^{\ld}X_{j_{r}}\ar[rr]
\ar[dd]_{\lambda_{1,\dotsc,1}\sma_{R}^{\ld}\dotsb \sma_{R}^{\ld}\lambda_{1,\dotsc,1}}^{\sim}
&&X_{j_{1}}\sma_{R}\dotsb\sma_{R}X_{j_{r}}
\ar[dd]^{\lambda_{1,\dotsc,1}\sma_{R}\dotsb \sma_{R}\lambda_{1,\dotsc,1}}\\
&X_{j}\ar[ur]^(.3){\lambda_{j_{1},\dotsc j_{r}}}_-{\sim}
\ar[dr]_(.3){\lambda_{1,\dotsc,1}}^-{\sim}\\
(X_{1}\sma_{R}\dotsb \sma_{R}X_{1})\sma_{R}^{\ld}\dotsb \sma_{R}^{\ld} 
(X_{1}\sma_{R}\dotsb \sma_{R}X_{1})\ar[rr]
&&X_{1}\sma_{R}\dotsb \sma_{R} X_{1}
}
\]
The maps labelled $\sim$ are isomorphisms in $\aD_{R}$ by the definition
of partial power system, and so we see that the vertical maps are
isomorphisms in $\aD_{R}$. It follows that 
the top horizontal map is an isomorphism in $\aD_{R}$ for all
$j_{1},\dotsc,j_{r}$ exactly when the bottom horizontal map is an
isomorphism in $\aD_{R}$ for all $j_{1},\dotsc,j_{r}$.  Looking at the
diagram
\[
\xymatrix{%
(X_{1}\sma_{R}\dotsb \sma_{R}X_{1})\sma_{R}^{\ld}\dotsb \sma_{R}^{\ld} 
(X_{1}\sma_{R}\dotsb \sma_{R}X_{1})\ar[r]
&X_{1}\sma_{R}\dotsb \sma_{R} X_{1}\\
(X_{1}\sma_{R}^{\ld}\dotsb \sma_{R}^{\ld}X_{1})
\sma_{R}^{\ld}\dotsb \sma_{R}^{\ld} 
(X_{1}\sma_{R}^{\ld}\dotsb \sma_{R}^{\ld}X_{1})\ar[u]\ar[r]_-{\iso}
&X_{1}\sma_{R}^{\ld}\dotsb \sma_{R}^{\ld} X_{1}\ar[u]
}
\]
we see that this happens exactly when $X$ is tidy.

Intuitively, tidy means that the $m$-th partial power $X_{m}$
is equivalent to the derived $m$-th smash power of $X_{1}$.  True
power systems need not be tidy in general, but if $X$ is cofibrant or
if we are working in the category of symmetric spectra, orthogonal
spectra, or EKMM $S$-modules and $X$ is cofibrant in the category of
$\gO$-algebras (for $\gO$ as above) or even cofibrant in the category
of commutative $R$-algebras, then the true power system $X_{m}=X^{m}$ 
is tidy.

To define a partial $\gO$-algebra structure on a partial power system
$A$, we need slightly more than the sequence of maps 
\[
\gO(m)_{+}\sma_{\Sigma_{m}}A_{m}\to A_{1},
\]
that suffices for a true $\gO$-algebra; rather, we need maps of the form
\[
(\gO(j_{1})\times\dotsb \times \gO(j_{r}))_{+}
\sma_{\Sigma_{j_{1}}\times \dotsb \times \Sigma_{j_{r}}}
A_{m}\to A_{r}
\]
for $j_{1}+\dotsb+j_{r}=m$.  The coherence is most easily expressed by
generalizing the monadic description of operadic algebras.

\begin{cons}\label{consmonad}
For an op-lax power system $X$, let 
\[
(\bgO X)_{m}=\bigvee_{j_{1},\dotsc,j_{m}}
(\gO(j_{1})\times \dotsb \times \gO(j_{m}))_{+}
\sma_{\Sigma_{j_{1}}\times \dotsb \times \Sigma_{j_{m}}}
X_{j_{1}+\dotsb+j_{m}}.
\]
This obtains a $\Sigma_{m}$-action by permuting the $j_{i}$'s and
performing the corresponding block permutation on
$X_{j_{1}+\dotsb+j_{m}}$. 
We make $\bgO X$ an op-lax power system using the structure maps
$\lambda$ of $X$ and the associativity isomorphism
\begin{multline*}
\bigvee_{j_{1},\dotsc,j_{m+n}}
(\gO(j_{1})\times \dotsb \times \gO(j_{m+n}))_{+}
\sma_{\Sigma_{j_{1}}\times \dotsb \times \Sigma_{j_{m+n}}}
(X_{j_{1}+\dotsb+j_{m}} \sma_{R} X_{j_{m+1}+\dotsb +j_{m+n}})
\\
\iso (\bgO X)_{m}\sma_{R}(\bgO X)_{n}.
\end{multline*}
\end{cons}

We have a natural map of op-lax power systems $X\to \bgO X$ induced by
the inclusion of the identity element in $\gO(1)$.  We have a natural
transformation of op-lax power systems
\[
\bgO \bgO X \to \bgO X
\]
induced by the operadic multiplication on $\gO$.  An easy check of
diagrams then proves the following proposition.

\begin{prop}
The structure above makes $\bgO$ a monad in the category of op-lax
power systems.
\end{prop}

Because we have assumed that each $\gO(m)$ is a free $\Sigma_{m}$-CW
complex, the smash products 
\[
(\gO(j_{1})\times \dotsb \times \gO(j_{m}))_{+}
\sma_{\Sigma_{j_{1}}\times \dotsb \times \Sigma_{j_{m}}}(-)
\]
preserve weak equivalences.  In particular, we see that
when $X$ is a partial power system, so is $\bgO X$.  Thus, we obtain
the following proposition.

\begin{prop}\label{propmonad}
Under the assumptions on $\gO$ above, $\bgO$ is a monad in the
category of partial power systems.
\end{prop}

Note that for a true power system $X_{m}=X^{(m)}$, we have 
\[
(\bgO X)_{m}\iso (\bgO X)_{1}^{(m)}=(\bgsO X)^{(m)},
\]
where $\bgsO$ is the usual monad in $R$-modules associated to the
operad $\gO$.
The monad $\bgO$ therefore generalizes to partial power systems the
monad of $R$-modules associated to the operad $\gO$.
We now have the following monadic definition of a partial
$\gO$-algebra, generalizing the monadic definition of a true
$\gO$-algebra. 

\begin{defn}\label{defpart}
A \term{partial $\gO$-algebra} (or \term{partial
$\alg{\gO}{R}$-algebra} when $R$ needs to be made explicit) is an
algebra over the monad $\bgO$ in partial power systems.
\end{defn}

In particular, a true $\alg{\gO}{R}$-algebra structure on $A$ is
precisely a partial $\alg{\gO}{R}$-algebra structure on the true power system
$A_{m}=A^{(m)}$.
We could define a more general notion of op-lax $\gO$-algebra in terms
of the op-lax power systems; we learned from Leinster that such an
algebra is precisely a strictly unital op-lax symmetric monoidal
functor from an appropriate 
category of operators associated to $\gO$.  A formulation of the
theory of partial $\gO$-algebras along these lines as well as
generalizations and further examples may be found in Leinster's
preprint \cite{LeinsterPartial}.

The remainder of the section discusses and reviews the proof of the
following theorem, the Kriz-May rectification theorem.  In it, we
write $T$ for the functor from $R$-modules to partial power systems
that takes an $R$-module $X$ to the true power system
$(TX)_{m}=X^{(m)}$. 

\begin{thm}[May \cite{MayMult}, Kriz-May \cite{kmbook}]
\label{thmrect}
For an operad $\gO$ as above, there exists a functor $R$ from partial
$\gO$-algebras to true $\gO$-algebras, an endofunctor $E^{\sharp}$ on
partial $\gO$-algebras, and natural weak equivalences of partial $\gO$-algebras
\[
\Id \from E^{\sharp}\to TR.
\]
Moreover, there exists an endofunctor $E$ on $\gO$-algebras and
natural weak equivalences of $\gO$-algebras
\[
\Id\from E \to RT.
\]
The functors and natural transformations are natural also in the
operad $\gO$.
\end{thm}

As a consequence, the categories of partial $\gO$-algebras
and of true $\gO$-algebras become equivalent after formally inverting the
weak equivalences. 

The proof of theorem is an application of the two-sided monadic bar
construction. To avoid confusion, write $UX$ for $X_{1}$ for a partial
power system $X$.  We note that for any partial power system $X$, the
structure maps $\lambda$ induce a natural map of partial power 
systems $X\to TUX$ (in fact, a weak equivalence), which together with
the identity $UTY=Y$, identify 
$U$ and $T$ as adjoints.  In this notation, the monad $\bgsO$ in
$R$-modules is $U\bgO T$.

We let 
$RA=B(\bgsO U,\bgO,A)$
be the geometric realization of the simplicial $R$-module
\[
B\subdot(\bgsO U,\bgO,A)= \bgsO U\,
\underbrace{\bgO\dotsb \bgO}_{\text{$\ssdot$ iterates}}\, A.
\]
Here the face maps are induced by the $\bgO$-action on $A$, the
monadic multiplication on $\bgO$, and the multiplication 
\[
\bgsO U\bgO X \to \bgsO \bgsO UX\to \bgsO UX.
\]
The degeneracy maps are induced by the unit of the monad $\bgO$.
Because the monad $\bgsO$ commutes with geometric realization, $RA$ is
naturally a $\gO$-algebra.

We let $E^{\sharp}A=B(\bgO,\bgO,A)$ be the geometric realization of
the ``standard resolution'' 
\[
B\subdot(\bgO,\bgO,A)=\bgO \,
\underbrace{\bgO\dotsb \bgO}_{\text{$\ssdot$ iterates}}\, A
\]
(where we understand the geometric realization of a partial power
system to be performed objectwise).
The augmentation $B\subdot(\bgO,\bgO,A)\to A$
is a map of simplicial partial $\gO$-algebras (where we regard $A$ as
a constant simplicial object) and a simplicial homotopy equivalence of
partial power systems.  Thus, the geometric realization is a natural
map of partial $\gO$-algebras and a weak equivalence (in fact, a
homotopy equivalence of $\Sigma_{m}$-equivariant $R$-modules in each
partial power). 

The natural map $\bgO X\to TU\bgO TU X = T\bgsO UX$ is a map of
partial $\gO$-algebras; it is a weak equivalence since it is induced by
the maps $\lambda_{m_{1},\dotsc,m_{r}}$. Moreover,
the map is compatible with the left action of the monad $\bgO$.  Thus,
we get a map of simplicial partial $\gO$-algebras 
\[
B\subdot(\bgO,\bgO,A)\to B\subdot(T\bgsO U,\bgO,A)
\]
which is a weak equivalence in each simplicial degree.  The geometric
realization is a weak equivalence and induces the natural weak equivalence
of partial $\gO$-algebras $E^{\sharp}\to TR$ in the statement.

On the $\gO$-algebra side, we take $E$ to be the geometric realization
of the standard resolution $B\subdot(\bgsO,\bgsO,-)$.  The
construction and study of the natural transformations are analogous
to the partial $\gO$-algebra case described in detail above.

\section[The Moore Algebra]%
{The Moore Partial Algebra of a Partial $\LC[1]$-Algebra}
\label{secmoore}

Moore constructed a variant of the loop space of a topological space
where the concatenation of loops is strictly associative and unital
and not just associative and unital up to homotopy.  This construction
easily extends to any $\LC[1]$ space, and in fact quite generally to
$\LC[1]$-algebras in topological categories.  In this section, we
generalize Moore's construction to the partial context, constructing a
partial associative $R$-algebra from a partial $\LCR[1]$-algebra.  In
the next section, we use this to construct the cyclic bar construction
of a partial $\LCR[1]$-algebra.

We begin by reviewing the construction of  the Moore algebra of a (true)
$\LC[1]$-algebra, before treating the slightly more complicated
partial case.  Recall that an element of $\LC[1](m)$ consists of an
ordered list of almost disjoint subintervals of the unit interval,
not necessarily in their natural order.  The operadic multiplication
$a\circ_{i}b$ replaces the $i$-th subinterval in $a$ with a scaled
version of the subintervals in $b$.  The element
\[
\gamma=([0,1/2],[1/2,1]) \in \LC[1](2)
\]
represents the loop multiplication for the action of $\LC[1]$ on a
loop space; the two composites 
\[
\gamma\circ_{1}\gamma=([0,1/4],[1/4,1/2],[1/2,1])
\quad\text{and}\quad 
\gamma\circ_{2}\gamma=([0,1/2],[1/2,3/4],[3/4,1])
\]
in $\LC[1](3)$ represent the
two associations of the multiplication of three loops in a loop space.
The basic idea of Moore's construction is to add 
a length parameter to build an associative multiplication from a
$\LC[1]$-multiplication.  Specifically, given lengths $r$ and $s$, we
get an element $\gamma_{r,s}$ in $\LC[1](2)$ that models the concatenation
of a loop of length $r$ with a loop of length $s$, rescaled to the
unit interval. 
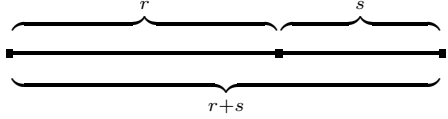
\[
\underbrace{%
\subseg{r}{10em}%
\subseg{s}{6em}}_{r+s} 
\]
Namely, $\gamma_{r,s}$ is the element
\[
\gamma_{r,s}=([0,r/(r+s)],[r/(r+s),1])\in \LC[1](2).
\]
Writing $P=(0,\infty)\subset \bR$ for
the space of positive real numbers, the ``concatenation formula'' 
\[
\Gamma \colon P\times P\to P\times \LC[1](2)
\]
sends $(r,s)$ to $(r+s,\gamma_{r,s})$.

For a $\LC[1]$-algebra $A$, we get an associative multiplication on
$P_{+}\sma A$ using $\Gamma$ and the $\LC[1](2)$-action,
\begin{multline*}
(P_{+}\sma A)\sma_{R}(P_{+}\sma A)\iso 
(P\times P)_{+}\sma (A\sma_{R}A)\to\\
(P\times \LC[1](2))_{+}\sma (A\sma_{R}A)\iso
P_{+}\sma (\LC[1](2)_{+}\sma (A\sma_{R}A))
\to P_{+}\sma A.
\end{multline*}
To make this unital, let $\bar P=[0,\infty)\subset \bR$ denote the
non-negative real numbers, and define the $R$-module $\MA$ by the
following pushout diagram.
\begin{equation}\label{eqmoore}
\begin{gathered}
\xymatrix@-1pc{%
P_{+}\sma R\ar[r]\ar[d]&P_{+}\sma A\ar[d]\\
\bar P_{+}\sma R\ar[r]&\MA
}
\end{gathered}
\end{equation}
The multiplication above then extends to an associative multiplication
on $\MA$, and is now unital with the unit $R\to \MA$ induced by the
inclusion of $\{0\}_{+}\sma R$ into $\bar P_{+}\sma R$.

\begin{prop}
For a $\LC[1]$-algebra $A$, $\MA$ is naturally an associative $R$-algebra.
\end{prop}

To relate $A$ and $\MA$, note that forgetting the lengths (collapsing
$P$ and $\bar P$ to a point), we obtain a map $\MA\to A$.  This map is
a homotopy equivalence of $R$-modules: The map $A\to \MA$ induced by
the inclusion of $1$ in $P$ provides the homotopy inverse.  The
composite map on $A$ is the identity and compatible null homotopies on $P$
and $\bar P$ induce a homotopy from the identity to the composite on
$\MA$.  (See \cite[\S6]{E1E2E3E4} for a comparison of $A$ and $\MA$ as
$\LC[1]$-algebras.)

To extend this to the case of partial algebras, we need to construct
an appropriate partial power system with the $m$-th partial power
analogous to 
the $m$-th smash power of the pushout in~\eqref{eqmoore}.  Let $\oT$
be the diagram with objects $a,b,c$ and arrows $\alpha,\beta$ as
pictured
\[
\xymatrix@-1pc{%
&c\ar[dl]_{\beta}\ar[dr]^{\alpha}\\b&&a
}
\]
so that a pushout is a colimit indexed on $\oT$.

\begin{cons}
Let $A$ be a partial $\LC[1]$-algebra.  We construct the op-lax power
system $\MA$ as follows.  Let $\MA[1]$ be the pushout
\[
\MA[1]=(P_{+} \sma A_{1})\cup_{P_{+}\sma R} (\bar P_{+}\sma R),
\]
where the map $R\to A_{1}$ is induced by the unique element of
$\LC[1](0)$.  Inductively, having defined
$\MA[1],\dotsc,\MA[m-1]$, we define $\MA[m]$ as a colimit over the
following diagram $D_{m}$ indexed on $\oT^{m}$.  At a spot indexed by
$(x_{1},\dotsc,x_{m})$, we put a copy of 
\[
(P_{x_{1}}\times \dotsb \times P_{x_{m}})_{+}\sma A_{\ma}
\]
where $\ma=\ma(x_{1},\dotsc,x_{m})$ is the number of occurrences of
$a$ in the $x_{i}$'s, and 
\[
P_{x_{i}}=\begin{cases}
P&x_{i}=a\text{ or }x_{i}=c\\
\bar P&x_{i}=b.
\end{cases}
\]
The map
\[
(x_{1},\dotsc,x_{i-1},c,x_{i+1},\dotsc,x_{m})
\to (x_{1},\dotsc,x_{i-1},y,x_{i+1},\dotsc,x_{m})
\]
induced by $\alpha$ (when $y=a$) is induced
by the map $A_{\ma-1}\to A_{\ma}$ induced by the element
\[
(\id,\dotsc,\id,1,\id,\dotsc,\id)\in 
\LC[1](1)^{\ma(x_{1},\dotsc,x_{i-1})}\times \LC[1](0)\times \LC[1](1)^{\ma(x_{i+1},\dotsc,x_{n})},
\]
where $\id$ is the identity element in $\LC[1](1)$ and $1$ is the
unique element in $\LC[1](0)$.  The map induced by $\beta$ (when
$y=b$) is induced by the inclusion of $P$ in $\bar P$ in the
appropriate factor.  We give $\MA[m]$ the $\Sigma_{m}$-action induced
by permuting the factors of $\oT^{m}$,
together with the appropriate re-arrangement of the $P_{x_{i}}$ factors and 
the $\Sigma_{\ma}$-action on $A_{\ma}$ (corresponding to restricting a
permutation in $\Sigma_{m}$ to the rearrangement of the $a$'s in the
object of $\oT^{m}$).   The structure maps of $A$ induce the structure
maps $\MA[m+n]\to \MA[m]\sma_{R}\MA[n]$ for $\MA$.
\end{cons}

Collapsing the $P_{x_{i}}$ factors to a point defines a map of op-lax
power systems $\MA\to A$.  An argument like the one above shows that
each map $\MA[m]\to A_{m}$ is a $\Sigma_{m}$-equivariant homotopy
equivalence of $R$-modules. 

\begin{prop}
Each map $\MA[m]\to A_{m}$ is a $\Sigma_{m}$-equivariant homotopy
equivalence of $R$-modules.  In particular, $\MA$ is a partial power
system and is tidy exactly when $A$ is.
\end{prop}

A partial associative $R$-algebra structure is a partial $\oA$-algebra
structure for $\oA$ the associative operad.  We write $\mu_{j}\in
\oA(j)$ for the canonical element, representing the 
$j$-fold multiplication (if $j>1$), the operadic identity (if $j=1$),
or the unit (if $j=0$).  To define a partial $R$-algebra structure
on the partial power system $X$, we need to specify maps
\[
(\mu_{j_{1}},\dotsc,\mu_{j_{m}})\colon X_{j_{1}+\dotsb +j_{m}}\to X_{m}
\]
for all $j_{1},\dotsc,j_{m}\geq 0$, satisfying the associativity and
identity diagrams implicit in Definition~\ref{defpart}.  These
conditions become easier to verify by thinking of the sequence
$j_{1},\dotsc,j_{m}$ as specifying a map of totally ordered sets
\[
\{1,\dotsc,j\}\to \{1,\dotsc,m\}
\]
where $j=j_{1}+\dotsb+j_{m}$: It specifies the unique weakly
increasing map $\phi$ such that the cardinality of $\phi^{-1}(i)$ is 
$j_{i}$.  The following well-known fact is a consequence of an easy
check of the diagrams, q.v.~\cite[VII\S5]{maclane}.  In it, 
we write $\otom$ for the totally ordered set
$\{1,\dotsc,m\}$, and $\oDDelta$ for the category whose objects are
the totally ordered sets $\otos[0],\otos[1],\dotsc $ and whose morphisms
are the weakly increasing maps.

\begin{prop}\label{propmaclane}
Let $X$ be a partial power system in the category of $R$-modules.  A
partial $R$-algebra structure on $X$ consists of a map $X_{\phi}\colon
X_{j}\to X_{m}$ for each $\phi \colon \otoj\to\otom$ in $\oDDelta$,
making $X$ a functor from $\oDDelta$ to $R$-modules, such that the
diagram
\[
\xymatrix{%
X_{i+j}\ar[r]^-{\lambda_{j,k}}\ar[d]_{X_{\theta+\phi }}
&X_{i}\sma_{R}X_{j}\ar[d]^{X_{\theta}\sma_{R}X_{\phi}}\\
X_{m+n}\ar[r]_-{\lambda_{m,n}}&X_{m}\sma_{R}X_{n}
}
\]
commutes for all $\theta \colon \otos[i]\to\otom$, $\phi\colon \otoj\to\oton$.
\end{prop}

More intrinsically, the previous proposition states that a partial
associative $R$-algebra is a strictly unital op-lax monoidal functor $(X,\lambda)$
from $\oDDelta$ to $R$-modules together with a $\Sigma_{m}$-action on
$X_{m}=X_{\otom}$ making $X$ into a partial power system,
cf.~\cite[1.6(a),2.2.1]{LeinsterPartial}. 

To construct the partial associative $R$-algebra structure on $\MA$,
first we must generalize the elements $\gamma_{r,s}$ above.
For $\vec x=(x_{1},\dotsc,x_{m})$ in $\oT^{m}$ and $\vec
r=(r_{1},\dotsc,r_{m})$ in $P_{x_{1}}\times \dotsb \times P_{x_{m}}$,
let $\gamma^{\vec x}_{\vec r}$ be the following element of $\LC[1](\ma)$,
where $\ma=\ma(x_{1},\dotsc,x_{n})$, the number of $a$'s among the
$x_{i}$.  If $\ma=0$, then $\gamma^{\vec x}_{\vec r}$ is the unique
element of $\LC[1](0)$.  Otherwise, define $k_{1} < k_{2} < \dotsb <
k_{\ma}$ by $x_{k_{i}}=a$, and note that $r_{k_{i}}>0$ for all $i$.  Let
$\gamma^{\vec x}_{\vec r}$ consist of the subintervals
\[
\left[\frac{r_{1}+\dotsb+r_{k_{i}-1}}{r_{1}+\dotsb+r_{m}},
\frac{r_{1}+\dotsb+r_{k_{i}}}{r_{1}+\dotsb+r_{m}}\right]
\]
(in their natural order) for $i=1,\dotsc,\ma$.  For example, if $m=\ma$
(i.e., all the $x_{i}'s$ are $a$'s), then $\gamma^{\vec x}_{\vec r}$
is the element the of $\LC[1](m)$ that
subdivides the unit interval into $m$ subintervals of lengths the
given proportions $r_{1},\dotsc,r_{m}$; when one of the $x_{i}$'s is
$c$ or $b$ with $r_{i}>0$, then $\gamma^{\vec x}_{\vec r}$ contains a
gap proportional to 
$r_{i}$ where the subinterval would have been if that $x_{i}$ were
$a$.  An $x_{i}$ that is $b$ with $r_{i}=0$, has neither a subinterval
nor a gap corresponding to it; in the multiplication below, it will
behave like a unit factor in an $m$-th smash power.

Next, for fixed $\phi \colon \otoj\to\otom$ in $\oDDelta$, let
$F_{\phi}$ be the functor from $\oT^{j}$ to $\oT^{m}$ that sends
$(x_{1},\dotsc,x_{j})$ to 
$(y_{1},\dotsc,y_{m})$ where 
\[
y_{i}=\begin{cases}
a&\text{if at least one $k\in \phi^{-1}(i)$ satisfies $x_{k}=a$}\\
b&\text{if every $k\in \phi^{-1}(i)$ satisfies $x_{k}=b$ (or $\phi^{-1}(i)$ is empty)}\\
c&\text{if at least one $k\in \phi^{-1}(i)$ satisfies $x_{k}=c$ and none satisfy $x_{k}=a$}\\
\end{cases}
\]
and does the only possible thing on maps. 

We defined $\MA[m]$ as the
colimit of a diagram $D_{m}$ indexed on $\oT^{m}$; we now define
$\MA[\phi]\colon \MA[j]\to\MA[m]$ to be the map on colimits induced by
a natural transformation $D_{\phi}\colon D_{j}\to F_{\phi}^{*}D_{m}$
as follows. 
For $(x_{1},\dotsc,x_{j})$ in $\oT^{j}$, 
define $J_{i}$ and $j_{i}$ by
$\phi^{-1}(i)=\{J_{i}+1,\dotsc,J_{i}+j_{i}\}$. (In other words, let
$j_{i}=|\phi^{-1}(i)|$ and $J_{i}=j_{1}+\dotsb+j_{i-1}$.)
write $\vec x_{\phi^{-1}(i)}$ for
$(x_{J_{i}+1},\dotsc,x_{J_{i}+j_{i}})$,
and for $(r_{1},\dotsc,r_{j})\in P_{x_{1}}\times \dotsb \times P_{x_{j}}$, write
$\vec r_{\phi^{-1}(i)}$ for $(r_{J_{i}+1},\dotsc,r_{J_{i}+j_{i}})$.
Then let $D_{\phi}$ be the
natural transformation 
\[
(P_{x_{1}}\times \dotsb \times P_{x_{j}})_{+}\sma A_{\ma(\vec x)}
\to (P_{y_{1}}\times \dotsb \times P_{y_{m}})_{+}\sma A_{\ma(\vec y)}
\]
(for $(y_{1},\dotsc,y_{m})=F_{\phi}(x_{1},\dotsc,x_{j})$ as above),
sending $(r_{1},\dotsc,r_{j})\in P_{x_{1}}\times \dotsb \times P_{x_{j}}$ to 
\[
\bigg(\sum \vec r_{\phi^{-1}(1)},\dotsc, \sum \vec r_{\phi^{-1}(m)}\bigg)
\in P_{y_{1}}\times \dotsb \times P_{y_{m}}
\]
(where we understand the sum to be zero when $\phi^{-1}(i)$ is empty), and sending
$A_{\ma(\vec x)}\to A_{\ma(\vec y)}$ by the map induced by
\[
\bigg(\gamma^{\vec x_{\phi^{-1}(k_1)}}_{\vec r_{\phi^{-1}(k_1)}},\dotsc
\gamma^{\vec x_{\phi^{-1}(k_\ell)}}_{\vec r_{\phi^{-1}(k_\ell)}}\bigg)
\in \LC[1](\ma(\vec x_{\phi^{-1}(k_1)} ))\times \dotsb \times \LC[1](\ma(\vec x_{\phi^{-1}(k_\ell)} )).
\]
where $k_1,...,k_\ell$ are the position of the $a$'s in $\vec y =
F_{\phi}(\vec x)$ and $\ell = \ma(\vec y).$ 

The formula in the $P_{x_{k}}$ factors lands in the appropriate
$P_{y_{i}}$ since $\sum \vec r_{\phi^{-1}(i)}$ can only be $0$ when
every object in the list $\vec x_{\phi^{-1}(i)}$ is $b$.

To see that $D_{\phi}$ are natural in $\oT^{j}$, it suffices to check
the maps $\alpha$ and $\beta$ separately.  For the maps $\beta$, the
formulas in the $P_{x_{k}}$ factors are clearly natural in $\oT^{n}$.
For the maps $\alpha$, we note that $\alpha$ is induced by $1\in
\LC[1](0)$ and the composition $\gamma^{\vec x}_{\vec r}\circ_{i}1$
drops the $i$-th subinterval.

This completes the construction of the natural transformation
$D_{\phi}$ and of the map $\MA[\phi]$.  A straightforward check of the
diagrams now proves the following theorem.

\begin{thm}\label{thmmoore}
The maps $\MA[\phi]$ above make $\MA$ into a partial associative $R$-algebra.
\end{thm}

\section{$THH$ of a Partial $\LC[n]$-Algebra}
\label{secncyc}

In this section, we study the multiplicative structure on $THH^{R}(A)$
for a partial $\LC$-algebra $A$.  Starting with the cyclic bar
construction of a partial associative $R$-algebra, we define
$THH^{R}(A)$ as the cyclic bar construction of the Moore partial
algebra $\MA$.  This depends only on the underlying
$\LC[1]$-structure; we use ``interchange'' (see
Definition~\ref{definterchange}) of the $\LC[1]$-structure with a
$\LC[n-1]$-structure to make $THH^{R}\subdot(\MA)$ into a partial
$\LC[n-1]$-algebra, proving Theorems~\ref{mainthh} and~\ref{mainrel}
from the introduction.  Finally, in Construction~\ref{consitthh}
below, we give a closed description of iterated $THH^{R}$.

We begin by reviewing the cyclic bar construction of a partial
associative $R$-algebra. 

\begin{cons}[The cyclic bar construction]\label{consncyalg}
For a partial associative $R$-algebra $A$, let $THH^{R}(A)$ be the
geometric realization of the cyclic $R$-module 
\[
THH^{R}_{p}(A)=A_{p+1}
\]
with action of the cyclic group $C_{p+1}$ inherited from the action of the
symmetric group $\Sigma_{p+1}$, 
with face maps $d_{i}$ induced by 
\[
(\id,\dotsc,\id,\mu,\id,\dotsc,\id)\in \oA(1)^{i}\times \oA(2)\times \oA(1)^{p-i-1}
\]
for $i=0,\dotsc,p-1$ (with $d_{p}$ induced by the cyclic permutation
followed by $d_{0}$), and degeneracy maps $s_{i}$ induced by
\[
(\id,\dotsc,\id,1,\id,\dotsc,\id)\in \oA(1)^{i}\times \oA(0)\times \oA(1)^{p-i}.
\]
\end{cons}

In the case of a partial $\LCR[1]$-algebra, we apply this construction
to the Moore partial algebra.

\begin{defn}\label{defncyen}
For a partial $\LCR[1]$-algebra $A$, we define $THH^{R}(A)=THH^{R}(\MA)$.
\end{defn}

We remark that this point-set construction does not generally
represent the correct homotopy type without additional assumptions on
$A$.  When $A$ is tidy (Definition~\ref{deftidy}), then $\MA$ is tidy,
and the simplicial object $THH^{R}\subdot(A)$ has the correct homotopy
type.  For the geometric realization to have the correct homotopy
type, it suffices for the simplicial object to be ``proper'' (for the
inclusion of the union of the degeneracies at each stage to be a
Hurewicz cofibration). The following proposition usually suffices for most
purposes.  We provide an iterable generalization in
Proposition~\ref{propitprop} at the end of this section.

\begin{prop}\label{proptrueprop}
Working in the context of $R$-modules
of symmetric spectra, orthogonal spectra, or EKMM $S$-modules, assume
that $A$ is a true $\LCR[1]$-algebra such that the map $R\to A$ is a
cofibration of the underlying $R$-modules.  Then
$\MA$ is tidy and $THH^{R}\subdot(\MA)$ is a proper simplicial object.
\end{prop}

\begin{proof}
Under the hypotheses above, each map $A^{m-1}\to A^{m}$ induced by
$\LC[1](0)$ is a cofibration of the underlying $R$-modules, and in the
case of symmetric spectra or orthogonal spectra, it follows that
(after passing to a retraction if necessary), each degeneracy map from
$THH^{R}_{p-1}(\MA)$ is the inclusion of a subcomplex in a fixed
relative $\oI'$-cell
complex \cite[5.4]{MMSS} structure on $THH^{R}_{p}(\MA)$ (relative to
the inclusion of $R$), where $\oI'=\oI$ is
the set of generating cofibrations in the model structure
\cite[6.2]{MMSS}.  In the context of EKMM 
$S$-modules, the same holds, but for $\oI'$ the set of generating
cofibrations together with the maps $R\sma S^{j-1}_{+}\to R\sma
B^{j-1}_{+}$ (for $j\geq 0$).   It follows that the union (colimit) of
these subcomplexes is a subcomplex and its inclusion is a Hurewicz
cofibration. 
\end{proof}

This completes the generalization of the cyclic bar construction to
a partial $\LC[1]$-algebra $A$.  The remainder of the section studies the
multiplicative structure when $A$ is a $\LC$-algebra.  In this case,
we understand $A$ to be a $\LC[1]$-algebra via the \term{first-coordinate
embedding} of $\LC[1]$ in $\LC[n]$: This embedding takes the sequence
of subintervals
\[
( [x_{1},y_{1}],\dotsc,[x_{m},y_{m}] ) \in \LC[1](m)
\]
to the sequence of subrectangles
\[
( [x_{1},y_{1}]\times [0,1]^{n-1},\dotsc ,
[x_{m},y_{m}] \times [0,1]^{n-1}) \in \LC(m).
\]
We also have a \term{last coordinates embedding} of $\LC[n-1]$ in
$\LC[n]$, taking the sequence of subrectangles
\[
( [x^{1}_{1},y^{1}_{1}]\times \dotsb \times [x^{n-1}_{1},y^{n-1}_{1}],
\dotsc,
[x^{1}_{m},y^{1}_{m}]\times \dotsb \times [x^{n-1}_{m},y^{n-1}_{m}])
\in \LC[n-1](m)
\]
to the sequence of subrectangles
\[
( [0,1]\times [x^{1}_{1},y^{1}_{1}]\times \dotsb \times [x^{n-1}_{1},y^{n-1}_{1}],
\dotsc,
[0,1]\times [x^{1}_{m},y^{1}_{m}]\times \dotsb \times [x^{n-1}_{m},y^{n-1}_{m}])
\in \LC(m).
\]
Both of these embeddings are special cases of the following ``interchange'' map.

\begin{defn}\label{definterchange}
For $\ell,m\geq 0$, the \term{interchange map} is the map 
\[
\rho \colon \LC[1](\ell)\times \LC[n-1](m)\to
\LC(\ell m)
\]
that takes the pair
\[
([a^{i},b^{i}]\mid 1\leq i\leq \ell),
([x^{j}_{1},y^{j}_{1}]\times \cdots \times [x^{j}_{n-1},y^{j}_{n-1}]
\mid 1\leq j\leq m)
\]
to the sequence of subrectangles of $[0,1]^{n}$,
\[
[a^{i},b^{i}]\times [x^{j}_{1},y^{j}_{1}]\times \cdots \times
[x^{j}_{n-1},y^{j}_{n-1}],
\]
(for $1\leq i\leq \ell$, $1\leq j\leq m$), ordered by lexicographical
order in $(i,j)$. 
\end{defn}

The first coordinate embedding is then 
\[
\embf(-)=\rho (-,([0,1]^{n-1}))\colon \LC[1]\to \LC[n]
\]
and the last coordinate embedding is
\[
\embl(-)=\rho(([0,1]),-)\colon \LC[n-1]\to \LC[n].
\]
Using $\embl$, for any element $c$ in $\LC[n-1](m)$, we get an element
$\embl(c)$ in $\LC(m)$ and hence a map $A_{m}\to A$. 

We call $\rho$ the interchange map because for a (true) $\LC$-algebra
$A$, for any $a$ in $\LC[1](\ell)$ and $c$ in $\LC[n-1](m)$,
both composites in the diagram
\begin{equation}\label{eqint}
\begin{gathered}
\xymatrix@C+2pc{%
(A^{(m)})^{(\ell)}\ar[r]^-{(\embl(c))^{(\ell)}}
\ar[d]_{a}
&A^{(\ell)}\ar[d]^{a}\\
A^{(m)}\ar[r]_-{\embl(c)}&A
}
\end{gathered}
\end{equation}
are the induced map of $\rho(a,c)\colon A^{(\ell m)}\to A$ under the
isomorphism $A^{(\ell m)}\iso (A^{(m)})^{(\ell)}$ induced by
lexicographical order.  In the diagram, the left vertical arrow is the
action of $a$ in the ``diagonal'' $\LC$-algebra structure on
$A^{(m)}$: Its has the structure map
\begin{equation}\label{eqdiagact}
\LC(\ell)_{+}\sma (A^{(m)})^{(\ell)}
\to
(\LC(\ell)_{+})^{(m)}\sma (A^{(m)})^{(\ell)}
\iso
(\LC(\ell)_{+}\sma A^{(\ell)})^{(m)}
\to A^{(m)}
\end{equation}
where the first map is induced by the diagonal map on $\LC(\ell)$,
the last map is the action map for $A$ on each of the $m$-factors, and
the isomorphism in the middle is the permutation $\sigma_{m,\ell}$
that rearranges the 
$\ell$ blocks of $m$ factors of $A$ into $m$ blocks of $\ell$ factors.
The following proposition is an immediate consequence of the
commutative diagram. 

\begin{prop}\label{proptrueint}
Let $A$ be a true $\LC$-algebra.  For any $c$ in $\LC[n-1](m)$, the
map $A^{(m)}\to A$ induced by 
$\embl(c)$ is a map of $\LC[1]$-algebras.
\end{prop}

The only obstacle to extending the previous proposition to partial
$\LC$-algebras is understanding $A_{m}$ as a partial $\LC$-algebra,
and the only obstacle here is understanding $A_{m}$ as a partial power
system.  We overcome these obstacles in the following definition.

\begin{defn}\label{defppspower}
For $X$ an op-lax power system and $m>0$, let $X^{[m]}$ be the op-lax
power system with $X^{[m]}_{p}=X_{pm}$, symmetric group action induced
by block permutation (with blocks of size $m$), and structure maps 
\[
\lambda_{p,q}\colon X^{[m]}_{p+q}\to X^{[m]}_{p}\sma_{R}X^{[m]}_{q}
\]
induced from the structure map $\lambda_{pm,qm}$ for $X$.  Then
$X^{[1]}=X$, and we let $X^{[0]}=R$.
\end{defn}

We note that when $X$ is a partial power system, $X^{[m]}$ is as well.
For a partial $\LC$-algebra $A$, the partial power system
$A^{[m]}$ obtains a ``diagonal'' partial $\LC$-algebra structure from
the $\LC$-action on $A$:  The element
\[
(c_{1},\dotsc,c_{j})\in \LC(\ell_{1})\times \dotsb \times \LC(\ell_{j})
\]
induces the map $A^{[m]}_{\ell}\to A^{[m]}_{j}$ (for
$\ell=\ell_{1}+\dotsb +\ell_{j}$) given by 
\begin{equation}\label{eqperm}
\sigma_{m,j}^{-1}\circ
(c_{1},\dotsc,c_{j},\dotsc,c_{1},\dotsc,c_{j})\circ
\sigma_{m,\ell}\colon A_{\ell m}\to A_{jm}
\end{equation}
(with $c_{1},\dotsc,c_{j}$ repeated $m$ times) where $\sigma_{m,k}$
denotes the permutation in $\Sigma_{km}$ that rearranges the $k$
blocks of $m$ into $m$ blocks of $k$, as 
in~\eqref{eqdiagact}.  Given an element $c$ in $\LC(m)$, we can also
make $c$ into a map of partial power systems $A^{[m]}\to A$,
by defining the map $c_{\ell}$ on the $\ell$-th partial power level to be
$(c,\dotsc,c)$, i.e., $c$ repeated $\ell$ times, without permutations.  Then
regarding 
\[
(a_{1},\dotsc,a_{j})\in \LC[1](\ell_{1})\times \dotsb \times \LC[1](\ell_{j})
\]
as a partial $\LC[1]$-algebra structure map and using the map of
partial power systems $\embl(c)$ for $c\in \LC[n-1](m)$, the
following interchange diagram commutes.
\[
\xymatrix@C+2pc{%
A^{[m]}_{\ell}\ar[d]_{(a_{1},\dotsc,a_{j})}
\ar[r]^{\embl(c)_{\ell}}
&A_{\ell}\ar[d]^{(a_{1},\dotsc,a_{j})}\\
A^{[m]}_{j}\ar[r]_{\embl(c)_{j}}
&A_{j}
}
\]
This is the partial analogue of~\eqref{eqint} and proves the following
partial analogue of Proposition~\ref{proptrueint}.

\begin{prop}\label{propint}
Let $A$ be a partial $\LC$-algebra.  For any $c$ in $\LC[n-1](m)$, the
map $A^{[m]}\to A$ induced by 
$\embl(c)$ is a map of $\LC[1]$-algebras.
\end{prop}

Returning to the cyclic bar construction, we now have the terminology
and notation to describe the multiplicative structure and prove
Theorems~\ref{mainthh} and~\ref{mainrel}.  Even in the
case when we start with a true $\LC$-algebra $A$, $THH^{R}(A)$ will
only have a partial multiplicative structure.  We construct the
op-lax power system as follows.

\begin{cons}\label{consthhps}
For a partial $\LCR[1]$-algebra $A$, we define the op-lax power
system $THH^{R}(A)$ by $(THH^{R}(A))_{m}=THH^{R}(\MoAlg{A^{[m]}})$.
\end{cons}

When each of the simplicial objects $THH^{R}\subdot(\MoAlg{A^{[m]}})$
is proper, this defines a partial power system, since geometric
realization then preserves the degreewise weak equivalences.  With
just this mild hypotheses, we can now prove the following version of
Theorem~\ref{mainrel}.

\begin{thm}\label{thmmainrel}
For $R$ a commutative
$S$-algebra, let $A$ be a partial $\LCR$-algebra such that the op-lax
power system $THH^{R}(A)$ is a partial power system. Then $THH^{R}(A)$
is a  partial $\LCR[n-1]$-algebra, naturally in maps of
$\LCR$-algebra maps in $A$.
\end{thm}

\begin{proof}
Applying Proposition~\ref{propint}, we see that every element $c$ in
$\LC[n-1](m)$ induces a map of partial associative $R$-algebras
\[
\MoAlg{A^{[m]}}\to \MA.
\]
More generally, the argument for Propositions~\ref{proptrueint}
and~\ref{propint} implies that elements $c_{1},\dotsc,c_{r}$ of
$\LC[n-1](j_{i})$ induce a map of partial $\LC[1]$-algebras
\[
A^{[j_{1}+\dotsb +j_{r}]}\to A^{[r]}
\]
and hence a map of partial associative $R$-algebras
\begin{equation}\label{eqactmoore}
\MoAlg{A^{[j_{1}+\dotsb+j_{r}]}}\to \MoAlg{A^{[r]}}.
\end{equation}
Restricting to the $p$ power and putting these maps together for all
elements of $\LC[n-1]$ at once, we get maps of $R$-modules
\begin{equation}\label{eqlevel}
(\LC[n-1](j_{1})\times \dotsb \times \LC[n-1](j_{r}))_{+}
\sma_{\Sigma_{j_{1}}\times \dotsb\times \Sigma_{j_{r}}}
\MoAlg{A^{[j_{1}+\dotsb+j_{r}]}}_{p}\to \MoAlg{A^{[r]}}_{p}.
\end{equation}
Because the maps~\eqref{eqactmoore} are maps of partial power systems,
the maps~\eqref{eqlevel} commute with the $\Sigma_{p}$-action.
Because the maps~\eqref{eqactmoore} are maps of partial associative
$R$-algebras, the maps~\eqref{eqlevel} commute with the simplicial
face maps when we view $\MoAlg{A^{[m]}}_{p}$ as
$(THH^{R}_{p-1}(A))_{m}$ (simplicial degree $p-1$ in the $m$-th
power).  Likewise, the maps~\eqref{eqlevel} commute with the
degeneracy operations.  Commuting the smash product and geometric
realization, we therefore get a map of partial power systems
\begin{equation}\label{eqact}
\bLC[n-1] THH^{R}(A) \to THH^{R}(A) .
\end{equation}
Using the associativity and unity of the
$\LC$-action on $A$ and the fact that $\embl$ is a homomorphism of
operads, a straightforward check  shows that \eqref{eqact} defines a
partial $\LC[n-1]$-algebra structure on $THH^{R}(A)$.
\end{proof}

As a special case, when $n=2$, $THH^{R}(A)$ is a $\LC[1]$-algebra.
Iterating $THH^{R}$ would mean applying $THH^{R}$ to the Moore partial
algebra $\MoAlg{THH^{R}(A)}$.  This partial associative $R$-algebra
admits a closed description in terms of the colimit diagrams that
construct the Moore partial algebra.  The simplification in terms of
the ``main'' submodules $P^{m}_{+}\sma A_{m}\subset \MA[m]$ captures
the main ideas while avoiding the complications.  Since geometric
realization commutes with the Moore partial algebra construction, it
is enough to describe the construction on each simplicial level.  In
these terms, the main submodule of $\MoAlg{THH^{R}_{p}(A)}_{m}$
(simplicial degree $p$, partial power $m$) is
\[
P^{(p+1)+m}_{+}\sma A_{(p+1)m}.
\]
Instead of having a length for each $A$ power, this rather has
a length for each column and each row, where we think of $A_{(p+1)m}$
as organized into $m$ rows of $p+1$ columns.  The main submodule of
$THH^{R}(THH^{R}(A))_{1}$ in bisimplicial degree $p,q$ then is
\[
P^{(p+1)+(q+1)}_{+}\sma A_{(p+1)(q+1)}.
\]
The face maps in the $p$-direction add the column lengths and
concatenate squares (in $\LC[2]$) in the horizontal direction, while
the face maps in the $q$-direction add the row length and concatenate
squares in the vertical direction.

More generally and precisely, for a partial $\LC$-algebra, we give the
following closed construction of the $n$-th iterate of $THH^{R}$

\begin{cons}\label{consitthh}
Let $A$ be a partial $\LC$-algebra. The $n$-th iterate of $THH^{R}$ is
isomorphic to the geometric realization
of the following $n$-fold simplicial set $T^{n}(A)$.  In
multisimplicial degree $p_{1},\dotsc,p_{n}$, $T^{n}(A)$ is the
colimit of the diagram indexed on $\oT^{(p_{1}+1)+\dotsb+(p_{n}+1)}$
that at the object
\[
\vec x=(x_{1},\dotsc,x_{q})=
(x^{1}_{0},\dotsc,x^{1}_{p_{1}},x^{2}_{0},\dotsc,x^{2}_{p_{2}},
\dotsc,x^{n}_{0},\dotsc,x^{n}_{p_{n}})
\]
is the $R$-module
\[
(P_{x_{1}}\times \dotsb \times P_{x_{q}})_{+}\sma 
A_{m_{1}\dotsb m_{n}},
\]
where $m_{i}=\ma(x^{i}_{0},\dotsc,x^{i}_{p_{i}})$ (the number of $a$'s
among the $x^{i}_{j}$'s).  The degeneracies are induced by inserting
$b$ in the appropriate spot in $\vec x$ with zero as the associated
length in $P_{b}=\bar P$.  In the $i$-th simplicial direction, the
$j$-th face map (for $j<p_{i}$) is induced as follows.  The lengths $(r,s)\in
P_{x^{i}_{j}}\times 
P_{x^{i}_{j+1}}$ add.  If $x^{i}_{j}=x^{i}_{j+1}=a$, we use the
element  
\[
c=[0,1]^{i-1}\times \gamma_{r,s}\times [0,1]^{n-i}
\]
of $\LC(2)$, applied $m_{1}\dotsb \hat m_{i}\dotsb m_{n}$ times, with
the appropriate permutations (as in~\eqref{eqperm}) to produce the map 
\[
A_{m_{1}\dotsb m_{n}}\to A_{m_{1}\dotsb (m_{i}-1)\dotsb m_{n}}.
\]
If one of $x^{i}_{j},x^{i}_{j+1}$ is $a$, we use the appropriate
element 
\begin{align*}
&([0,1]^{i-1} \times [0,r/(r+s)]\times [0,1]^{n-i}),
\quad \text{or}\\
&([0,1]^{i-1} \times [r/(r+s),1]\times [0,1]^{n-i})
\end{align*}
of $\LC(1)$.  We use the identity on the $A$'s factor if neither of
$x^{i}_{j},x^{i}_{j+1}$ is $a$.  The $p_{i}$-th 
face map is the appropriate permutation followed by the zeroth face map.
\end{cons}

Using the concrete construction above, we can see that the
multisimplicial object $T^{n}\subdot(A)$ has the correct homotopy type
when $A$ is tidy, and hence in this case a thickened realization will
capture the correct homotopy type for iterated $THH^{R}$.  Moreover,
we can see that $T^{n}\subdot(A)$ is proper under reasonable
hypotheses on $A$ like the one discussed above or its iterable
generalization, which we now discuss.

For the iterable generalization of
Proposition~\ref{proptrueprop}, we use the
following  technical hypothesis.  In it, $\oI'$ is as in the proof of 
Proposition~\ref{proptrueprop}: In the context of symmetric spectra or
orthogonal spectra, $\oI'$ is the collection of generating
cofibrations in the model structure; in the context of EKMM
$S$-modules, $\oI'$ also includes 
the maps $R\sma S^{j-1}_{+}\to R\sma B^{j}_{+}$.

\begin{hyp}[Technical hypothesis on a partial {$\LC[1]$}-algebra $A$]
\label{hypprop}
For all $m$, the maps $R\to A_{m}$ are
relative $\oI'$-cell complexes (for $\oI'$ as above) such that each of the
$m$ maps 
$A_{m-1}\to A_{m}$ (induced by the action of $\LC[1](0)$) is the
inclusion of a relative subcomplex.
\end{hyp}

\begin{prop}\label{propitprop}
Working in the category of $R$-modules
of symmetric spectra, orthogonal spectra, or EKMM $S$-modules, let
$A$ be a partial $\LCR[n]$-algebra satisfying Hypothesis~\ref{hypprop}.
Then
$\MA$ is tidy, $THH^{R}\subdot(\MoAlg{A^{[m]}})$ is a proper simplicial
object for all $m$, and the op-lax power system $THH^{R}(A)$ is a
partial power system.
Moreover, if $n\geq 2$, then the partial
$\LCR[n-1]$-algebra $THH^{R}(A)$
also satisfies Hypothesis~\ref{hypprop}.
\end{prop}

\begin{proof}
Since $\LC(0)$ is a point, the maps $A_{m-1}\to A_{m}$ induced by
$\LC[1](0)$ (using the first coordinate embedding) coincide with the
maps $A_{m-1}\to A_{m}$ induced by $\LC[n-1](0)$ (using the last
coordinates embedding). The proof is now a straightforward cell argument
in terms of the construction of $\MA$ and $THH^{R}$.
\end{proof}

We now have 
Theorem~\ref{mainrel} as an immediate corollary of the previous
proposition and Theorem~\ref{thmmainrel}.
Theorem~\ref{mainthh} is the special case $R=S$.

\section{The Bar Construction for Augmented $\LC$-Algebras} 
\label{secbar}

In this section we study the reduced version of $THH$, usually called
the bar construction,  defined for an augmented partial $\LC$-algebra.
We begin with the definition of augmented partial $\LC$-algebras and
augmented partial associative $R$-algebras.

\begin{defn}
An \term{augmented} partial $\LCR$-algebra consists of a partial $\LCR$-algebra $A$
together with a map of partial $\LCR$-algebras $\epsilon \colon A\to
R$ called the \term{augmentation}.  Likewise, an augmented partial
associative $R$-algebra consists of a partial associative $R$-algebra
$A$ and a map of partial associative $R$-algebras $A\to R$.
\end{defn}

In general for an op-lax power system $X$, given a map $X\to R$, we
get a pair of maps $X_{m}\to X_{m-1}$ as
composites
\[
\begin{gathered}
e_{1}\colon X_{m}\to X_{1}\sma_{R}X_{m-1}\to R\sma_{R}X_{m-1}\iso X_{m-1}\\
e_{m}\colon  X_{m}\to X_{m-1}\sma_{R}X_{1}\to X_{m-1}\sma_{R}R\iso X_{m-1}
\end{gathered}
\]
using $\lambda_{1,m-1}$ and $\lambda_{m-1,1}$.  We think of $e_{1}$
and $e_{m}$ as applying the augmentation to the first and last spots
in $X_{m}$.  Note that $e_{m}$ can
be obtained from $e_{1}$ and the permutation actions on $X_{m}$ and
$X_{m-1}$, and using permutations like this, we can define analogous
maps $e_{2},\dotsc,e_{m-1}$ that apply the augmentation to an
arbitrary spot in $X_{m}$.

In the case of an augmented partial associative $R$-algebra $A$, the
maps $e_{j}$  make the following diagram commute.
\[
\xymatrix{%
A_{m}\ar[r]^{e_{1}}\ar[d]_{(\mu,\id,\dotsc,\id)}&A_{m-1}\ar[d]^{e_{1}}
&&A_{m}\ar[r]^{e_{m}}\ar[d]_{(\id,\dotsc,\id,\mu)}&A_{m-1}\ar[d]^{e_{m-1}}\\
A_{m-1}\ar[r]_{e_{1}}&A_{m-2}&&A_{m-1}\ar[r]_{e_{m-1}}&A_{m-2}
}
\]
We use these maps in the following construction.

\begin{cons}\label{consbar}
For an augmented partial associative $R$-algebra $A$, let $BA$ be the
geometric realization of the simplicial $R$-module $B\subdot A$ that is $A_{m}$ in
simplicial degree $m$, with degeneracy maps $s_{i}$ induced by the action
of $\oA(0)$, with face maps $d_{i}$ for $1\leq i\leq m-1$ induced
by the action of $\mu \in \oA(2)$, and with face maps $d_{0}=e_{1}$ and
$d_{m}=e_{m}$ as defined above.
\end{cons}

As in the previous section, the construction may not have the correct
homotopy type without the additional assumption that $A$ is tidy and
an additional assumption ensuring that the simplicial object is proper.

For $A$ a partial augmented $\LC[n]$-algebra, we set $BA=B\MA$, and we
extend $BA$ to a partial power system by setting 
$BA_{m}=B(A^{[m]})=B(\MoAlg{A^{[m]}})$.  The trick used in the proof of
Theorems~\ref{mainthh} and~\ref{mainrel} now constructs the
$\LC[n-1]$-structure in the following theorem.

\begin{thm}\label{thmbarmult}
Let $A$ be an augmented partial $\LC$-algebra $A$ such that $B\subdot
A^{[m]}$ is a proper simplicial object for all $m$.  Then the bar
construction 
$BA$ is naturally an augmented partial
$\LC[n-1]$-algebra. 
\end{thm}

The properness hypothesis holds in particular when $A$ satisfies
Hypothesis~\ref{hypprop} or the hypothesis of Proposition~\ref{proptrueprop}. 
The augmentation in the theorem is the map $BA\to R$ induced by the
augmentations 
$\MoAlg{A^{[m]}}_{p}\to R$.   We now give a detailed
description of the iterated bar construction along the lines of
Construction~\ref{consitthh}.

\begin{cons}
Let $A$ be an augmented partial $\LC$-algebra. The $n$-th iterate of
the bar construction is 
isomorphic to the geometric realization
of the following $n$-fold simplicial set $B^{n}(A)$.  In
multisimplicial degree $p_{1},\dotsc,p_{n}$, $B^{n}(A)$ is the
colimit of the diagram indexed on $\oT^{p_{1}+\dotsb+p_{n}}$
that at the object
\[
\vec x=(x_{1},\dotsc,x_{q})=
(x^{1}_{1},\dotsc,x^{1}_{p_{1}},x^{2}_{1},\dotsc,x^{2}_{p_{2}},
\dotsc,x^{n}_{1},\dotsc,x^{n}_{p_{n}})
\]
is the $R$-module
\[
(P_{x_{1}}\times \dotsb \times P_{x_{q}})_{+}\sma 
A_{m_{1}\dotsb m_{n}},
\]
where $m_{i}=\ma(x^{i}_{1},\dotsc,x^{i}_{p_{i}})$ (the number of $a$'s
among the $x^{i}_{j}$'s).  The degeneracies are induced by inserting
$b$ in the appropriate spot in $\vec x$ with zero as the associated
length in $P_{b}=\bar P$.  In the $i$-th simplicial direction, the
$j$-th face map (for $0<j<p_{i}$) is induced as follows.  The lengths $(r,s)\in
P_{x^{i}_{j}}\times 
P_{x^{i}_{j+1}}$ add.  If $x^{i}_{j}=x^{i}_{j+1}=a$, we use the
element  
\[
c=[0,1]^{i-1}\times \gamma_{r,s}\times [0,1]^{n-i}
\]
of $\LC(2)$, applied $m_{1}\dotsb \hat m_{i}\dotsb m_{n}$ times, with
the appropriate permutations (as in~\eqref{eqperm}) to produce the map 
\[
A_{m_{1}\dotsb m_{n}}\to A_{m_{1}\dotsb (m_{i}-1)\dotsb m_{n}}.
\]
If one of $x^{i}_{j},x^{i}_{j+1}$ is $a$, we use the appropriate
element 
\begin{align*}
&([0,1]^{i-1} \times [0,r/(r+s)]\times [0,1]^{n-i}),
\quad \text{or}\\
&([0,1]^{i-1} \times [r/(r+s),1]\times [0,1]^{n-i})
\end{align*}
of $\LC(1)$.  We use the identity on the $A$'s factor if neither of
$x^{i}_{j},x^{i}_{j+1}$ is $a$.  The $0$-th and $p_{i}$-th 
face maps are obtained by application of the appropriate maps $e_{j}$.
\end{cons}

The category of $R$-modules under and over $R$ has $R$ as both an
initial and final object.  As a consequence, this category admits a
tensor with based spaces: For $M$ an $R$-module under and over $R$ and
$X$ a based space, the tensor of $M$ with $X$, $M\hotimes X$, is
formed as a pushout
\[
\xymatrix{%
R\sma X_{+}\cup_{R}M \ar[r]\ar[d]&M\sma X_{+}\ar[d]\\
R\ar[r]&M\hotimes X.
}
\]
In particular, we have a suspension in the category of $R$-modules
under and over $R$ that we denote as $\Sigma_{R}$.  For an augmented
partial $R$-algebra $A$, the first partial power $A_{1}$ is an $R$-module under
and over $R$, as is the first partial power $\MA[1]$ of the Moore
partial algebra.  The unit maps $\MA[1]\to \MA[m]$ induce maps 
\[
\underbrace{\MA[1]\cup_{R}\dotsb \cup_{R}\MA[1]}_{m\text{ factors}}
\to \MA[m],
\]
which together induce a map of simplicial objects
\[
\MA[1]\hotimes S^{1}\subdot \to B\subdot A,
\]
which on geometric realization induces a map $\Sigma_{R}\MA[1]\to BA$,
natural in $A$.  Using the explicit description of the multisimplicial
object $B^{n}A$ above, the simplicial map above generalizes to a
multi-simplicial map
\[
\MA[1]\hotimes (S^{1}\subdot\sma \dotsb \sma S^{1}\subdot)\to 
B^{n}_{\ssdot,\dotsc,\ssdot}A
\]
as follows.  Thinking of an element of $S^{1}_{p}$ as an element of
the based set $\{0,\dotsc,p\}$, a non-basepoint element of
$S^{n}_{p_{1},\dotsc,p_{n}}$ is an $n$-tuple $\vec
j=(j_{1},\dotsc,j_{n})$ with $1\leq j_{i}\leq p_{i}$.  We have one
copy of $\MA[1]$ in $\MA[1]\hotimes S^{n}_{p_{1},\dotsc,p_{n}}$ for
each $\vec j$, which we map into $B^{n}_{p_{1},\dotsc,p_{n}}A$ by a
map induced by a map of diagrams.  For fixed $\vec j$, we have the
functor $\oT\to \oT^{p_{1}+\dotsb +p_{n}}$ sending $x$ to the object
$\vec x$ where $x^{i}_{j_{i}}=x$ and $x^{i}_{k}=b$ for $k\neq j_{i}$.
We use the map of diagrams compatible with this functor sending
$P_{+}\sma A_{1}$ (or $\bar P_{+}\sma R$ or $P_{+}\sma R$, for $x=a$,
$b$, or $c$, respectively) into $(P_{x_{1}}\times\dotsb \times
P_{x_{q}})_{+}\sma A_{1}$ (or $(P_{x_{1}}\times\dotsb \times
P_{x_{q}})_{+}\sma R$) induced by the identity on $A_{1}$ (or $R$) and
sending $r$ in $P$ to
$(r^{1}_{1},\dotsc,r^{n}_{p_{n}})$ in $P_{x^{1}_{1}}\times\dotsb
\times P_{x^{n}_{p_{n}}}$ with $r^{i}_{j_{i}}=r$ and $r^{i}_{k}=0$ for
$k\neq j_{i}$.  This then describes a map from $\MA[1]\hotimes
S^{n}_{p_{1},\dotsc,p_{n}}$ to $B^{n}_{p_{1},\dotsc,p_{n}}A$ that
respects the face and degeneracy maps in all simplicial direction.  On
geometric realization, we get a map
\begin{equation}\label{eqsusp}
\Sigma^{n}_{R}\MA[1]\to B^{n}A,
\end{equation}
natural in $A$.  The map of $R$-modules under and over $R$ from
$\MA[1]$ to $A_{1}$ (obtained by forgetting length coordinate) induces
a map
\[
\Sigma^{n}_{R}\MA[1]\to \Sigma^{n}_{R}A_{1}
\]
which is a weak equivalence when the unit $R\to A_{1}$ is a Hurewicz
cofibration, so in particular, when the hypothesis of
Theorem~\ref{thmbarmult} holds.

\section{The Bar Construction for Non-Unital Partial $\LC$-algebras}
\label{secbar2}

In the next section, we study the reduced Andr\'e--Quillen cohomology
of a $\LCR$-algebra; as we will see there, this is easiest when we
work with ``non-unital'' $\LCR$-algebras in place of augmented
$\LCR$-algebras.  For this reason, we take this section to redo the
work of the previous section in the non-unital context.  We begin with
the definition of a non-unital partial $\LCR$-algebra.

\begin{defn}\label{defnonunital}
Let $\NC$ be the operad with $\NC(0)$ empty and $\NC(m)=\LC(m)$ for
$m>0$.  A \term{non-unital} partial $\LC$-algebra is a partial $\NC$-algebra.
\end{defn}

If $N$ is a non-unital partial $\LC$-algebra, then we can form an associated
augmented partial $\LC$-algebra by formally adding a unit.  Noting that for an
$R$-module $X$,
\[
(X\vee R)^{m}=\bigvee_{s\subset \otom} X^{(s)},
\]
(smash power indexed on the set $s$)
where $\otom=\{1,\dotsc,m\}$,
the $R$-modules
\[
KN_{m}=\bigvee_{s\subset \otom} N_{|s|}
\]
naturally form a partial power system with $KN_{1}=N_{1}\vee R$.  The
partial $\NC$-algebra structure on $N$ extends to a partial
$\LC$-algebra structure on $KN$, with the
missing elements in the subsets acting like place-holders for the unit
$1\in \LC(0)$ and with the action of $\LC(0)$
manipulating the sets $\otom=\{1,\dotsc,m\}$ and their subsets.
Specifically, on the summand corresponding to $s\subset \otom$,
the element 
\[
(c_{1},\dotsc,c_{j})\in \LC(m_{1})\times \dotsb \times
\LC(m_{j})
\]
(for $m_{1}+\dotsb +m_{j}=m$) induces the map
\begin{equation}\label{eqnuact}
(c'_{1},\dotsc,c'_{j'})\colon N_{|s|}\to N_{|t|}
\end{equation}
into the $t\subset \otoj=\{1,\dotsc,j\}$ summand, as follows.
Noting that $(c_{1},\dotsc,c_{j})$ has $m$ total inputs, we form 
\[
(c'_{1},\dotsc,c'_{j'})\in \LC(m'_{1})\times \dotsb \times \LC(m'_{j'})
\]
(with $m'_{1}+\dotsb+m'_{j'}=|s|$) by plugging $1\in \LC(0)$ into each
input that corresponds to an element not in $s$, and then dropping any $c_{i}$
whose inputs all become plugged.  The subset $t$ then consists of those
elements $i$ in $\otoj$ where $c_{i}$ was not dropped.

In the case of true algebras, we can also go from augmented
$\LC$-algebras to non-unital $\LC$-algebras: For a true $\LC$-algebra
$A$, we can form the true non-unital $\LC$-algebra $N$ as the homotopy
pullback of the augmentation, $N=R^{I}\times_{R}A$.  In the next
section, we will see that for true $\LC$-algebras, up to homotopy,
working with non-unital $\LC$-algebras is equivalent to working with
augmented $\LC$-algebras (Theorem~\ref{thmaugqe}).

We have a corresponding notion of non-unital partial associative $R$-algebra,
and the construction $K$ above also defines a functor from non-unital
partial associative $R$-algebras to augmented partial associative $R$-algebras.
We can generalize the Moore algebra to this context.
For a non-unital partial $\LC[1]$-algebra $N$, the power system
\[
\MN[m]=P^{m}_{+}\sma N_{m}
\]
forms a non-unital partial associative $R$-algebra using the
length concatenation construction in the Moore algebra.  To compare
this with $\MoAlg{KN}$, note that 
\[
K\MN[m]=\bigvee_{s\subset \otom}
P^{s}_{+}\sma N_{|s|},
\]
(the cartesian product of copies of $P$ indexed on $s$), while
\[
\MoAlg{KN}_{m}=\bigvee_{s\subset \otom}
(P^{m}_{s})_{+}\sma N_{|s|},
\]
where $P^{m}_{s}$ is the subset of $\bar P^{m}=[0,\infty)^{m}$ of
points $(r_{1},\dotsc, r_{m})$ with $r_{i}>0$ if $i\in s$.
The inclusion of $P^{s}$ in $P^{m}_{s}$ as the subset where $r_{i}=0$
for $i\not\in s$ induces a map $K\MN$ to $\MoAlg{KN}$ that is a map of
partial associative $R$-algebras and a $\Sigma_{m}$-equivariant
homotopy equivalence of $R$-modules in each partial power.

We can construct a version of the bar construction for a non-unital
$\LC$-algebra, taking 
\[
B_{m} N = \MN[m]
\]
for the non-unital Moore algebra construction above.  Since $N$ and
$\MN$ do not have units, this collection of $R$-modules does not admit
degeneracy maps, but the face maps in Construction~\ref{consbar} still
make sense (using the trivial map for the zeroth and last face map in
place of the augmentation).
Then $B\subdot N$ forms a $\Delta$-object (a simplicial 
object without degeneracies), and we can form $BN$ as the
geometric realization (by gluing $B_{m}\sma \Delta[m]_{+}$ along the
face maps).  We compare $BN$ and $BKN$ in the following
proposition. 

\begin{prop}
The $R$-modules $BN$ and $B(K\MN)$ are canonically isomorphic and the
map of $R$-modules $B(K\MN)\to B(\MoAlg{KN})=BKN$ is a homotopy
equivalence. 
\end{prop}

\begin{proof}
Associated to any $\Delta$-object we obtain a simplicial
object by formally attaching degeneracies; the geometric realization
of the $\Delta$-object is canonically isomorphic to the geometric
realization  of the associated simplicial object.  In this case, the
simplicial object associated to $BN$ is in simplicial degree $m$
the $R$-module 
\[
\bigvee_{j\leq m}\ \bigvee_{f\colon \tom\to\toj} \MN[j],
\]
where the maps $f\colon \tom\to\toj$ in the inner wedge are the
iterated degeneracies in the category $\DDelta$, i.e., the weakly
increasing epimorphisms of totally  
ordered sets $\{0,\dotsc,m\}\to\{0,\dotsc,j\}$.  We have a one-to-one
correspondence between the set of such epimorphisms $f$ and the set of $j$
element subsets $s$ of $\{1,\dotsc,m\}$ (where the elements in $s$ are
the first elements of $\{0,\dotsc,m\}$ that $f$ sends to each of the
elements $1,\dotsc,j$ of $\{0,\dotsc,j\}$), and this defines an
isomorphism between the $R$-module above and $K\MN[m]$.  As $m$
varies, these isomorphisms preserve the face and degeneracy maps in
the simplicial objects. This proves the first statement.  For the
second statement, we note that the homotopy equivalences $K\MN[p]\to
\MoAlg{KN}_{p}$ (using linear homotopies) commute with the face and
degeneracy maps in the bar construction and so extend to homotopy
equivalences on the geometric realizations.  
\end{proof}

We note that the properness issues that plague the discussion of the
cyclic bar construction and the bar construction for augmented
algebras disappear for non-unital algebras: For the simplicial version
of the construction $BN$,
each degeneracy is the inclusion of a wedge summand and the inclusion
of the union of degeneracies likewise is the inclusion of a wedge
summand.  For the construction $BKN$, each
degeneracy is induced by 
smashing with the inclusion of $\{0\}_{+}$ in $\bar
P_{+}=[0,\infty)_{+}$ followed by the inclusion of a wedge summand,
and the inclusion of the union of the degeneracies admits a similar
description on each of its wedge summands.  Thus, in particular, 
we have the following version of Theorem~\ref{thmbarmult}.

\begin{thm}\label{thmbarmultnu}
For a non-unital $\LC$-algebra $N$, $BKN$ is a naturally an augmented
partial $\LC[n-1]$-algebra. 
\end{thm}

We have a variant of $BN$ where we use the trivial $R$-module in
place of $R$ at the zero level.  We denote this as $\tB N$, 
and we identify the partial power system $B(K\MN)$
as $K\tB N$.
Specifically, $(\tB N)_{m}$ is the geometric realization of the
simplicial object which in degree $p$ is
\[
(\tB_{p}N)_{m}=
\bigvee_{s_{1},\dotsc,s_{m}} 
P^{|s_{1}\cup\dotsb \cup s_{m}|}_{+}\sma 
N_{|s_{1}|+\dotsb +|s_{m}|},
\]
with the wedge over the $m$-tuples of non-empty subsets of
$\otop=\{1,\dotsc,p\}$.  We can identify $(\tB_{p}N)_{m}$ as a
submodule 
of 
\[
(B_{p}KN)_{m}=\bigvee_{s_{1},\dotsc,s_{m}} 
(P^p_{s_{1}\cup \dotsb \cup s_{m}})_{+}\sma 
N_{|s_{1}|+\dotsb +|s_{m}|}
\]
with the wedge over the $m$-tuples of all subsets of
$\otop$.  From the work above we have a partial
$\LC[n-1]$-algebra structure on the partial power system $B_{p}KN$;
this restricts to a non-unital partial
$\LC[n-1]$-algebra structure on the partial power system $\tB_{p}N$ as
follows.

The action becomes easier to describe when we re-index the summands
by writing
\[
(B_{p}KN)_{m}=\bigvee_{s_{1},\dotsc,s_{p}} 
(P(s_{1})\times \dotsb \times P(s_{p}))_{+}\sma 
N_{|s_{1}|+\dotsb +|s_{p}|}
\]
where the sum is over the $p$-tuples of subsets of
$\otom=\{1,\dotsc,m\}$ and 
$P(s)=\bar P$ if $s$ is empty and $P(s)=P$ if $s$ is non-empty; the
relationship between these two indexings of the wedge sum 
corresponds to arranging $\{1,\dotsc ,pm\}$ into the $p$ blocks of $m$
elements $(1,\dotsc,m),(m+1,\dotsc, 2m),\dotsc, ((p-1)m+1,\dotsc,pm)$
in the first description or the $m$ blocks of $p$ elements
$(1,m+1,2m+1,\dotsc,(p-1)m+1),\dotsc, (m,2m,\dotsc ,pm)$ in the new
description.  Letting $P'(s)=\{0\}$ 
if $s$ is empty and $P'(s)=P(s)=P$ if $s$ is non-empty, then in this formulation,
\[
(\tB_{p}N)_{m}=
\bigvee_{s_{1},\dotsc,s_{p}} 
(P'(s_{1})\times \dotsb \times P'(s_{p}))_{+}\sma 
N_{|s_{1}|+\dotsb +|s_{p}|},
\]
where the sum is over the $p$-tuples of subsets of $\{1,\dotsc,m\}$
such that $s_{1}\cup\dotsb \cup s_{p}=\{1,\dotsc,m\}$.
In the $\LC[n-1]$-action, for $m_{1}+\dotsb +m_{j}=m$, $m_{i}>0$, an
element 
\[
(c_{1},\dotsc,c_{j})\in \LC[n-1](m_{1})\times \dotsb \times \LC[n-1](m_{j})
\]
acts on the $s_{1},\dotsc,s_{p}$ summand by the identity map on the $P$
factors and by the map
\[
(c^{1}_{1},\dotsc,c^{1}_{k_{1}},\dotsc,c^{p}_{1},\dotsc,c^{p}_{k_{p}})
\in \LC[n-1](m^{1}_{1})\times \dotsb \times \LC[n-1](m^{p}_{k_{p}})
\]
from $N_{|s_{1}|+\dotsb +|s_{p}|}$ to $N_{|t_{1}|+\dotsb +|t_{p}|}$
(using the last coordinates embedding of $\LC[n-1]$ in $\LC[n]$),
where $t_{i}$ and $(c^{i}_{1},\dotsc,c^{i}_{k_{i}})$ are as
in~\eqref{eqnuact} above:  $(c^{i}_{1},\dotsc,c^{i}_{k_{i}})$ is
formed by plugging $1\in \LC[n-1](0)$ into the input corresponding to
elements of $\{1,\dotsc,m\}$ not in $s_{i}$, dropping any $c_{r}$
where all inputs are plugged; $t_{i}$ consists of the indexes $r$
where $c_{r}$ is not dropped when forming
$(c^{i}_{1},\dotsc,c^{i}_{k_{i}})$.  To see that this action restricts
to $\tB_{p}N$, we just need to observe that when $s_{1}\cup\dotsb \cup
s_{p}=\{1,\dotsc ,m\}$, then $t_{1}\cup\dotsb \cup t_{p}=\{1,\dotsc ,j\}$.

Just as for $B_{p}KN$, the $\LC[n-1]$-action on $\tB_{p}N$ is compatible with the
face and degeneracy maps in the simplicial object $\tB N$.  This makes
$\tB N$ into a non-unital partial $\LC[n-1]$-algebra with the map $K\tB
N\to BKN$ a map of partial $\LC[n-1]$-algebras.

\begin{thm}\label{thmbarnu}
For a non-unital partial $\LC$-algebra $N$, the bar construction $\tB
N$ is naturally a non-unital $\LC[n-1]$-algebra and the weak equivalence 
$K\tB N\to B(KN)$ is a natural map of partial
$\LC[n-1]$-algebras.
\end{thm}

Next we describe the iterated bar construction $\tB^{n}N$.  Using the
first description of the partial power system $\tB_{p}N$ above, we see that
$\tB^{2}N$ is the $\Delta$-object in simplicial $R$-modules which in
degree $p,q$ is 
\[
\MoAlg{\tB_{p}N}_{q}=
P^{q}_{+}\sma \bigl(\bigvee_{s_{1},\dotsc,s_{q}} 
P^{|s_{1}\cup\dotsb \cup s_{q}|}_{+}\sma 
N_{|s_{1}|+\dotsb +|s_{q}|}\bigr),
\]
with the wedge over the $q$-tuples of non-empty subsets of
$\otop=\{1,\dotsc,p\}$. In bidegree $p,q$, the associated bisimplicial
$R$-module is then
\[
\tB^{2}_{p,q}N=
\bigvee_{\putatop{t\subset \{1,\dotsc,q\}}{t\neq\{\}}}\ 
\bigvee_{s_{1},\dotsc,s_{|t|}} 
P^{|s_{1}\cup\dotsb \cup s_{|t|}|+|t|}_{+}\sma 
N_{|s_{1}|+\dotsb +|s_{|t|}|},
\]
with the inside wedge over the $|t|$-tuples of non-empty subsets of
$\otop$.  If we re-index the subsets $s_{i}$ by the elements
of $t$ and set $s_{i}=\{\}$ for $i\not\in t$, then we get 
\[
\tB^{2}_{p,q}N=
\bigvee_{\putatop{s_{1},\dotsc,s_{q}\subset \otop}%
{s_{1}\cup\dotsb \cup s_{q}\neq \{\}}}
P^{|s_{1}\cup\dotsb \cup s_{q}|+n(s_{1},\dotsc,s_{q})}_{+}\sma 
N_{|s_{1}|+\dotsb +|s_{q}|},
\]
where $n(s_{1},\dotsc,s_{q})$ denotes the number of
$s_{1},\dotsc,s_{q}$ that are non-empty.
Finally, the collections $s_{1},\dotsc,s_{q}\subset \{1,\dotsc p\}$
satisfying $s_{1}\cup\dotsb \cup s_{q}\neq \{\}$ are in one to one
correspondence with the non-empty subsets of 
\[
\otop\times \otos[q]=\{1,\dotsc,p\}\times \{1,\dotsc,q\}.
\]
For $s\subset \otop\times \otos[q]$, write $\n^{p,q}_{1}(s)$ for the
subset of $\otop$ of elements $i$ such that $s\cap (\{i\}\times
\otos[q])$ is non-empty and likewise write  $\n^{p,q}_{2}(s)$ for the
subset of $\otos[q]$ of elements $i$ such that $s\cap 
(\otop\times \{i\})$ is non-empty.  In other words, $\n^{p,q}_{1}$ and
$\n^{p,q}_{2}$ are the images of the projection maps from
$\otop\times \otos[q]$ to $\otop$ and $\otos[q]$, respectively.  Now 
we can identify $\tB^{2}_{p,q}N$ as 
\[
\tB^{2}_{p,q}N=
\bigvee_{\putatop{s\subset \otop\times \otos[q]}{s\neq\{\}}}
(P^{\n^{p,q}_{2}(s)}\times P^{\n^{p,q}_{1}(s)})_{+}\sma N_{|s|}.
\]
Working inductively, we see that in multisimplicial degree
$p_{1},\dotsc,p_{n}$, 
\begin{equation}\label{eqconsitbar}
\tB^{n}_{p_{1},\dotsc,p_{n}}N=
\bigvee_{\putatop{s\subset \otop_{1}\times\dotsb 
\times  \otop_{n}}{s\neq\{\}}}
(P^{\n^{p_{1},\dotsc,p_{n}}_{n}(s)}\times \dotsb \times 
P^{\n^{p_{1},\dotsc,p_{n}}_{1}(s)})_{+}\sma N_{|s|},
\end{equation}
where $\n^{p_{1},\dotsc,p_{n}}_{i}(s)$ is the subset of elements $j$ in
$\otop_{i}$ such that 
\[
s\cap (\otop_{1}\times \dotsb \times \otop_{i-1}
\times \{j\}\times \otop_{i+1}\times \dotsb \times \otop_{n})
\]
is non-empty.

For $B^{n}KN$, we obtain a completely analogous description, but with 
$s=\{\}$ in the wedge sum and with different length coordinates.
Recall that for $u\subset \otop$, 
$P^{p}_{u}$ denotes the subset of $\bar P^{p}=[0,\infty)^{p}$ of elements
$(r_{1},\dotsc,r_{p})$ where $r_{j}>0$ for all $j\in u$.  Then the
length coordinates on the summand indexed by $s$ is
$P^{p_{i}}_{\n^{p_{1},\dotsc,p_{n}}_{i}(s)}$.
In other words,
\[
B^{n}_{p_{1},\dotsc,p_{n}}KN=
\bigvee_{s\subset \otop_{1}\times\dotsb 
\times  \otop_{n}}
(P^{p_{n}}_{\n^{p_{1},\dotsc,p_{n}}_{n}(s)}\times \dotsb \times 
P^{p_{1}}_{\n^{p_{1},\dotsc,p_{n}}_{1}(s)})_{+}\sma N_{|s|}.
\]

To complete the closed description of the iterated bar construction, we
describe the face and degeneracy maps.  The degeneracy map $s_{j}$ in
the $i$-th simplicial direction is induced by the map $\otop_{i}$
to $\{1,\dotsc,p_{i}+1\}$ sending $1,\dotsc,j-1$ by the identity and
$j,\dotsc,p_{i}$ by $m\mapsto m+1$.  

For $0<j<p_{i}$ the face map
$d_{j}$ adds the appropriate pair
$r,s\mapsto r+s$ in the $P$ factors (corresponding to $j,j+1\in
\otop_{i}$), and performs the action
\[
c=\id^{i-1}\times \gamma_{r,s}\times \id^{n-i}\in \LC(2)
\]
on the corresponding spots in 
$N_{|s|}$: We first use a permutation to rearrange from
lexicographical order on $s$ to the lexicographical order where the
$i$-th index is least significant.  Then for fixed $a_{k}\in
\otop_{k}, k\neq i$, in this order, the elements
\[
(a_{1},\dotsc,a_{i-1},j,a_{i+1},\dotsc,a_{n})
\qquad \text{and}\qquad
(a_{1},\dotsc,a_{i-1},j+1,a_{i+1},\dotsc,a_{n})
\]
are adjacent when they both appear in $s$.  For such elements, we
apply $c$ on the appropriate spot on $N_{|s|}$, but when one is
missing, we plug $1\in \LC(0)$ in that input of $c$ and apply that
element of $\LC(1)$ to the spot in $N_{|s|}$. (When neither element is
in $s$, no action needs to be taken for that pair.) We then use the
permutation to rearrange back to the natural lexicographical order.

The zeroth face map $d_{0}$ is the trivial map on the summands indexed
by those subsets $s$ where $1\in \n^{p_{1},\dotsc,p_{n}}_{i}(s)$.
Fixing $s$ with $1\not\in \n^{p_{1},\dotsc,p_{n}}_{i}(s)$, the
action of $d_{0}$ sends this summand to 
summand indexed by $s'$, where $s'$ is obtained
by subtracting $1$ from the $i$-th coordinate of each element of $s$.
On the length factors, the first factor gets dropped and for $j\ge 1$
the $j+1$ coordinate becomes the new $j$ 
coordinate.  On the $N$ factor, $|s|=|s'|$ and $N_{|s|}$ maps by the
identity to $N_{|s'|}$.  
The last face map $d_{p_{i}}$ has a similar description: it is the
trivial map on the summands 
indexed by those subsets $s$ where $p_{i}\in
\n^{p_{1},\dotsc,p_{n}}_{i}(s)$ and on those summands
where $p_{i}\not\in \n^{p_{1},\dotsc,p_{n}}_{i}(s)$, it lands in $s\subset \otop_{1}\times \dotsb
\times \otop'_{i}\times \dotsb \times \otop_{n}$ (for
$p_{i}'=p_{i}-1$), dropping the last length coordinate and acting by the identity on $N_{|s|}$.

We note from the description above that we have a canonical inclusion
of $\Sigma^{n}N_{1}$ into $\tB^{n}N$: Using the singleton subsets of
$\otop_{1}\times \dotsb \times \otop_{n}$ and the element $1\in
P$, we get a map of multisimplicial $R$-modules  $N_{1}\sma
S^{1}\subdot \sma \dotsb 
\sma S^{1}\subdot\to \tB^{n}_{\ssdot,\dotsc,\ssdot}N$, which
on geometric realization induces a map
\[
\Sigma^{n}N_{1}\to \tB^{n}N.
\]
We have an 
inclusion of $KN_{1}=R\vee N_{1}$ in $M(KN)_{1}$ where we send the $R$
factor in as length zero and the $N_{1}$ factor in as length one.
This defines a map in the category of $R$-modules under and over $R$,
splitting the usual map $M(KN)_{1}\to KN_{1}$ (induced by forgetting
lengths).  Noting that $\Sigma^{n}_{R}KN_{1}=R\vee \Sigma^{n}N_{1}$,
we have the following commutative diagram, relating the map above to
the map~\eqref{eqsusp}.
\begin{equation}\label{eqsuspsplit}
\begin{gathered}
\xymatrix{%
&R\vee \Sigma^{n}N_{1}\ar[r]\ar[d]\ar[ld]_{=}&R\vee \tB^{n}N\ar[d]\\
\Sigma^{n}_{R}KN_{1}&\Sigma^{n}_{R}MKN_{1}\ar[l]^{\simeq}\ar[r]&B^{n}KN
}
\end{gathered}
\end{equation}

\section%
{Reduced Topological Quillen Homology and the Iterated Bar Construction} 
\label{sectqh}

In this section we relate the iterated bar construction for augmented
$\LC$-algebras of the previous section to reduced topological Quillen
homology, proving
Theorem~\ref{maintq}. Quillen homology
theories are defined in terms of derived indecomposables, and (as
shown in~\cite{mbthesis}) work best in the context of non-unital
algebras. We begin 
by reviewing the Quillen equivalence of augmented and non-unital
algebras.

Recall from the previous section the functor $K$ from (true)
non-unital $\LC$-algebras to (true) augmented $\LC$-algebras that
wedges on a unit, $KN=R\vee N$.  This functor is left adjoint to the
functor $I$ from augmented $\LC$-algebras to non-unital $\LC$-algebras
that takes the (point-set) fiber of the augmentation map,
$IA=*\times_{R}A$.  We have a Quillen closed model structure on each
of these categories where the fibrations and weak equivalences are the
underlying fibrations and weak equivalences of $R$-modules; the
cofibrations are the retracts of $\btNC \oI$-cell complexes (in the
non-unital case) and of $\btLC \oI$-cell complexes (in the augmented
case) where $\oI$ is the set of generating cofibrations.  The functor
$I$ preserves fibrations and acyclic fibrations, and so the adjoint
pair $K,I$ forms a Quillen adjunction. Since we can calculate the
effect on homotopy groups of $K$ on arbitrary non-unital
$\LC$-algebras and of $I$ on fibrant augmented $\LC$-algebras, we see
that when $A$ is fibrant, a map of augmented $\LC$-algebras $KN\to A$
is a weak equivalence if and only if the adjoint map $N\to IA$ is a
weak equivalence; it follows that $K,I$ is a Quillen equivalence.

\begin{thm}\label{thmaugqe}
The functors $K$ and $I$ form a Quillen equivalence between the
category of non-unital $\LC$-algebras and the category of augmented
$\LC$-algebras. 
\end{thm}

Given an $R$-module $M$, we can make $M$ into a non-unital
$\LC$-algebra by giving it the trivial $\LC$-action, letting
\[
\LC(m)_{+}\sma_{\Sigma_{m}}M^{(m)}\to M
\]
be the trivial map for $m>1$ and the composite $\LC(1)_{+}\sma M\to
*_{+}\sma M\iso M$ for $m=1$.  This defines a functor $Z$ (the ``zero
multiplication'' functor) from $R$-modules to non-unital
$R$-algebras.  The functor $Z$ has a left adjoint ``indecomposables''
functor $Q$, which can be constructed as the coequalizer
\[
\xymatrix@C-1pc{%
\displaystyle 
\bigvee_{m>0}\LC(m)_{+}\sma_{\Sigma_m}N^{(m)}\mathstrut
\ar[r]<-.5ex>\ar[r]<.5ex>
&N\ar[r]&QN,
}
\]
where one map is the action map for $N$ and the other map is the
zero multiplication action map.  Since the functor $Z$ preserves fibrations and
weak equivalences, we see that $Q,Z$ forms a Quillen adjunction.

\begin{thm}\label{thmzq}
The functors $Q$ and $Z$ form a Quillen adjunction between the
category of non-unital $\LC$-algebras and the category of $R$-modules.
\end{thm}

For an augmented $\LC$-algebra $A$ and an $R$-module $M$, one can
define the reduced topological Quillen cohomology groups in terms of
derivations of $A$ with coefficients in $M$, i.e., as maps in the
homotopy category of augmented $\LC$-algebras from $A$ to $KZM$ (or
$KZ\Sigma^{*}M$).  Applying 
the Quillen equivalence of Theorem~\ref{thmaugqe} and the Quillen
adjunction of Theorem~\ref{thmzq}, we can identify this as maps in the
derived category of $R$-modules from $QN$ to $M$ (or $\Sigma^{*}M$),
where $N$ is a cofibrant non-unital $\LC$-algebra with $KN$ equivalent
to $A$.  In other words, the left derived functor $\TQN$ of $Q$
produces an object representing topological Quillen homology.
We write $\TQ$ for the composite of $\TQN$ with the 
right derived functor of $I$. 

\begin{defn}\label{deftqh}
For an augmented $\LC$-algebra $A$,
the $\LC$-algebra cotangent complex at the augmentation is the
$R$-module of derived indecomposables $\TQ$.
\end{defn}

We have a canonical natural map $\tB^{n}N\to \Sigma^{n}QN$ defined as
follows.  Thinking of $\Sigma^{n}QN$ as the geometric realization of
the multisimplicial $R$-module $QN \sma (S^{1}\subdot)^{n}$, we send
the summand indexed by $s\subset 
\otop_{1}\times \dotsb \times \otop_{n}$ in~\eqref{eqconsitbar} by
the counit of the $Q,Z$ adjunction $N\to QN$ (and dropping the $P$
factors) when $|s|=1$ and by the trivial map when $|s|>1$. This
clearly commutes with the degeneracy maps and commutes with the face
maps since every face map that changes the cardinality of the indexing
set is either the trivial map or lands in the decomposables. 

Using the constructions above, Theorem~\ref{maintq} becomes the
following theorem stated in terms of non-unital $\LC$-algebras.

\begin{thm}\label{thmtq}
The natural map $\tB^{n}N\to \Sigma^{n}QN$ is a weak equivalence when 
$N$ is a cofibrant non-unital $\LC$-algebra.
\end{thm}

To prove this theorem, we use the monadic bar
construction trick from~\cite[\S5]{mbthesis}.  The key observation is
that both the functors $\tB^{n}$ and $Q$ commute with geometric
realization.  This is clear from the construction for $\tB^{n}$, but
follows for $Q$ because $Q$ is a topological left adjoint and because
geometric realization of simplicial operadic algebras can be formed as
a topological colimit in the category of algebras; see
\cite[VII\S3]{ekmm}.  In particular, applied to the monadic bar construction
$B(\btNC,\btNC,N)$, we get isomorphisms
\[
\tB^{n}B(\btNC,\btNC,N)\iso B(\tB^{n}\btNC,\btNC,N)
\quad \text{and}\quad 
QB(\btNC,\btNC,N)\iso B(Q\btNC,\btNC,N).
\]
The argument now simplifies from \cite{mbthesis} since in our context
a cofibrant non-unital $\LC$-algebra is cofibrant as an $R$-module.
Regarding $B(\btNC,\btNC,N)$ as a topological colimit in non-unital
$\LC$-algebras, it follows that $B(\btNC,\btNC,N)$ is cofibrant as a
non-unital $\LC$-algebra when $N$ is.  Since the simplicial objects 
\[
B\subdot(\tB^{n}\btNC,\btNC,N)
\qquad \text{and}\qquad 
B\subdot(Q\btNC,\btNC,N).
\]
are always proper (the inclusion of the union of the degeneracies in
each simplicial degree is induced by the inclusion of $\id$ in
$\LC(1)$ and the inclusion of a wedge summand),
Theorem~\ref{thmtq} now reduces to the following lemma.

\begin{lem}\label{lemqt}
For any cofibrant $R$-module $X$, the natural  map $\tB^{n}\btNC X\to
\Sigma^{n}Q\btNC X$ is a weak equivalence.
\end{lem}

The functor $\Sigma^{n}Q\btNC X$ is canonically isomorphic to the
$n$-th suspension functor $\Sigma^{n}$; moreover, the natural
transformation  $\tB^{n}\btNC \to \Sigma^{n}$ is split by the natural
transformation
\[
\Sigma^{n}\to \Sigma^{n}\btNC  \to \tB^{n}\btNC,
\]
induced by the inclusion of the singleton subsets
in~\eqref{eqconsitbar}.   Thus, to prove Lemma~\ref{lemqt}, it
suffices to prove the following lemma.

\begin{lem}\label{lemreverse}
For any cofibrant $R$-module $X$, the natural map $\Sigma^{n}X\to
\tB^{n}\btNC X$ is a weak equivalence. 
\end{lem}

We can make a further reduction by identifying $\tB^{n}\btNC X$ as the
functor associated to a symmetric sequence.  Applying the
explicit description of $\tB^{n}$ in the previous section, we note
that the face and degeneracies in $\tB^{n}\btNC X$ preserve
homogeneous degrees in $X$.  We can therefore decompose $\tB^{n}\btNC$
naturally into a wedge sum of its homogeneous pieces,
\[
\tB^{n}\btNC X= \bigvee_{m> 0} \SSB(m)\sma_{\Sigma_{m}}X^{(m)},
\]
where each $\SSB(m)$ is a based $\Sigma_{m}$-space; in fact, each
$\SSB(m)$ is a free based $\Sigma_{m}$-cell complex since it is the
geometric realization of a multisimplicial $\Sigma_{m}$-space that in
each multisimplicial degree is a wedge of pieces of the form
\[
\left(
(\bar P^{i} \times P^{j})\times 
(\LC(m_{1})\times\dotsb \times \LC(m_{r}))
\times_{\Sigma_{m_{1}}\times \dotsb \times
\Sigma_{m_{r}}}\Sigma_{m}
\right)_{+}.
\]
The natural map
$\Sigma^{n}X\to \tB^{n}\btNC X$ is induced by the inclusion of $S^{n}$
in $\SSB(1)$.   Thus, it suffices to show that the map $S^{n}\to
\SSB(1)$ induces a weak equivalence on $R$-homology and that each
$\SSB(m)$ has trivial $R$-homology.  

Although it is not hard to show directly that the map $S^{n}\to
\SSB(1)$ is a weak equivalence, analysis of the construction of
$\SSB(m)$ is rather complicated for a direct argument
(cf.~\cite[\S 8]{FresseItBar})
and we take a shorter oblique approach in terms of the functor 
$\tB^{n}\btNC$.  Let $R^{0}_{c}$ be a cofibrant $R$-module equivalent
to $R$, and consider the set with $m$ elements,
$\otom=\{1,\dotsc,m\}$.   We note that
$\SSB(m)\sma_{\Sigma_{m}}(\otom_{+})^{(m)}$ contains $\SSB(m)$ as a
wedge summand, and so 
\[
\pi_{*}(\tB^{n}\btNC(R^{0}_{c}\sma \otom_{+}))
\]
contains the $R$-homology $R_{*}\SSB(m)$ as a direct summand.
Rewriting in terms of the natural transformation  $\Sigma^{n}\to
\tB^{n}\btNC$, we have reduced Lemma~\ref{lemreverse} to the following
lemma.

\begin{lem}\label{lemfs} 
The natural map
$\Sigma^{n}(R^{0}_{c}\sma X_{+})\to 
\tB^{n}\btNC (R^{0}_{c}\sma X_{+})$ is a weak equivalence for every
finite set $X$.
\end{lem}

Lemma~\ref{lemfs} follows from the
analogous statement in terms of augmented algebras, that the map
\[
R\vee \Sigma^{n}(R^{0}_{c}\sma X_{+})\to
B^{n}\btLC (R^{0}_{c}\sma X_{+})
\]
is a weak equivalence for all finite sets $X$.  We have an isomorphism
of $\LC$-algebras
\[
\btLC (R \sma X_{+})\iso R\sma (\btLC X)_{+}
\]
where $\btLC X$ is the free $\LC$-space on $X$.  The weak equivalence
$R^{0}_{c}\to R$ induces a weak equivalence of $\LC$-algebras
\[
\btLC(R^{0}_{c}\sma X_{+})\to \btLC(R\sma X_{+})\iso
R\sma (\btLC X)_{+},
\]
but this is not a map of augmented $\LC$-algebras: The augmentation on
the left is induced by the trivial map $X_{+}\to *$, but the
augmentation on the right is induced by the trivial map $\btLC X\to
*$.  In terms of $\btLC(R\sma X_{+})$, the right-hand augmentation is
induced by $X_{+}\to S^{0}$ and the $\LC$-action $\btLC R\to R$.
Since we are free to choose any cofibrant model $R^{0}_{c}$, choosing
one that is a suspension, we can construct a map 
\[
\alpha \colon R^{0}_{c}\sma X_{+}\to 
R^{0}_{c} \vee R^{0}_{c}\sma X_{+}
\]
that represents the sum of the map $X_{+}\to S^{0}$ and the identity on
$X_{+}$ smashed with $R^{0}_{c}$.  Write $\epsilon$ for the composite
map 
\[
R^{0}_{c}\sma X_{+}\to R^{0}_{c}\to R,
\]
which is homotopic to (but probably not equal to) the map induced by
$X_{+}\to S^{0}$, and choose a homotopy $h$.   Using $\alpha$, we get
a map of $\LC$-algebras 
\[
\bar \alpha \colon \btLC(R^{0}_{c}\sma X_{+})\to 
\btLC(R^{0}_{c}\sma X_{+}),
\]
which respects augmentations when we give the copy on the left the
augmentation 
induced by $\epsilon$; an easy filtration argument shows that this map
is a weak equivalence.  Now we get a diagram of weak equivalences of
augmented $\LC$-algebras
\[
\xymatrix{%
\btLC(R^{0}_{c}\sma X_{+}\sma\{0\}_{+})\ar[r]\ar[d]_{\bar \alpha}
&\btLC(R^{0}_{c}\sma X_{+}\sma I_{+})
&\btLC(R^{0}_{c}\sma X_{+}\sma \{1\}_{+})\ar[l]\ar[d]\\
\btLC(R^{0}_{c}\sma X_{+})
&&R\sma (\btLC X)_{+}
}
\]
which respects augmentations when we use the augmentation induced by
$X_{+}\to S^{0}$ on the right and the augmentation induced by $h$ in
the middle (on $\btLC(R^{0}_{c}\sma X_{+}\sma I_{+})$).  Since
applying $B^{n}$ preserves these weak equivalences, we get that 
\[
B^{n}\btLC(R^{c}_{0}\sma X_{+})
\qquad\text{and}\qquad
B^{n}(R \sma (\btLC X)_{+}).
\]
are weakly equivalent.

The map in Lemma~\ref{lemfs} is induced by the inclusion of
$R^{0}_{c}\sma X_{+}$ in $\btNC (R^{0}_{c}\sma X_{+})$ and the
section~\eqref{eqsuspsplit} 
of the natural map~\eqref{eqsusp}
\[
\Sigma^{n}_{R}M(\btLC (R^{0}_{c}\sma X_{+})) \to B^{n}\btLC
(R^{0}_{c}\sma X_{+}). 
\]
We can follow the natural map~\eqref{eqsusp} along the diagram of weak
equivalences between 
$B^{n}\btLC(R^{c}_{0}\sma X_{+})$ and $B^{n}(R \sma (\btLC X)_{+})$
above, and lift the map from 
\[
R\vee \Sigma^{n}(R^{0}_{c}\sma X_{+}) = \Sigma^{n}_{R}(R^{0}_{c}\sma X_{+})
\]
up to
homotopy all the way around to a map
\[
R\vee \Sigma^{n}(R^{0}_{c}\sma X_{+})\to 
\Sigma^{n}_{R}M(R \sma (\btLC X)_{+})\to
B^{n}(R \sma (\btLC X)_{+}).
\]
Specifically, the map from $R\vee \Sigma^{n}(R^{0}_{c}\sma
X_{+})$ to each of 
\[
\begin{small}
\xymatrix@C-1pc@R-1pc{%
\Sigma^{n}_{R}M\btLC(R^{0}_{c}\sma X_{+}\sma\{0\}_{+})\ar[d]\ar[r]
&\Sigma^{n}_{R}M\btLC(R^{0}_{c}\sma X_{+}\sma I_{+})
&\Sigma^{n}_{R}M\btLC(R^{0}_{c}\sma X_{+}\sma \{1\}_{+})\ar[l]\ar[d]\\
\Sigma^{n}_{R}M\btLC(R^{0}_{c}\sma X_{+}),
&&\Sigma^{n}_{R}M(R\sma (\btLC X)_{+})
}
\end{small}
\]
induces a weak equivalence on the submodules in homogeneous filtration
one (and below).  In particular, on the bottom right, looking at the map 
\[
R\vee \Sigma^{n}(R^{0}_{c}\sma X_{+})\to 
\Sigma^{n}_{R}M(R \sma (\btLC X)_{+})
\simeq
\Sigma^{n}_{R}(R \sma (\btLC X)_{+})=
R\sma (\Sigma^{n}(\btLC X))_{+},
\]
we now see that the map in Lemma~\ref{lemfs} is a weak equivalence 
if and only if the map 
\[
R\sma (\Sigma^{n}X_{+})_{+}\to
B^{n}(R \sma (\btLC X)_{+})
\]
induced by the inclusion of $X$ in $\btLC X$ is a weak equivalence.
This reduces Lemma~\ref{lemfs}  to the following lemma.

\begin{lem}\label{lemswitch}
The natural map
$R\sma \Sigma^{n}(X_{+})_{+}\to 
B^{n}(R \sma (\btLC X)_{+})$ is a weak equivalence for every 
finite set $X$.
\end{lem}

The point of this lemma is that it lets us compare with the classical
bar construction on spaces.  The construction $B^{n}$ in the previous
section used little about the category of $R$-modules and generalizes
to an iterated bar construction on the category of partial
$\LC$-spaces (where we use the cartesian product to define partial
power systems).  In fact, in spaces, it is much easier to describe
because we can talk in terms of  elements.  For a $\LC$-space $A$, each
element of the Moore construction $\MA$ has a length, and we make the
bar construction
\[
B\subdot A=B\subdot(\MA)=\MA\times \dotsb \times \MA
\]
into a partial power system by insisting that the lengths match up: We
take the $m$-th partial power $(B_{p} A)_{m}$ to be the
subset of $(B_{p}A)^{m}=(\MA^{p})^{m}$ where the length vectors for
each of the $m$ copies of $\MA^{p}$ all agree.  We have an entirely
similar description when $A$ is a partial $\LC$-space (noting that
elements of $\MA[\ssdot]$ also have sequences of lengths).  We also
note that when $A=\Omega Z$ with its
$\LC[1]$-structure coming from the standard $\LC[1]$-structure on the
loop space, then by construction, $\MA$ is the Moore loop space
$\Omega_{M}Z$. 

The functor $R\sma(-)_{+}$ from unbased spaces to $R$-modules takes
partial power systems to partial power systems.  By inspection, for
any partial $\LC$-space $Z$, we have an isomorphism of partial
$\LC[n-1]$-algebras
\[
B(R\sma Z_{+})\iso R\sma (BZ)_{+},
\]
and iterating, an isomorphism $B^{n}(R\sma Z_{+})\iso R\sma
(B^{n}Z)_{+}$.  Finally, Lemma~\ref{lemswitch} is a consequence of the
following proposition.

\begin{prop}\label{propdeloop}
For any finite set $X$, the map $\Sigma^{n} X_{+}\to B^{n}\btLC X$ is a weak
equivalence. 
\end{prop}

To prove the proposition, we inductively analyze $B^{i}\btLC X$.  The
inclusion $X\to \Omega^{n} \Sigma^{n} X_{+}$ induces a map of $\LC$-spaces
$\btLC X\to \Omega^{n} \Sigma^{n} X_{+}$.  By the group completion theorem,
this map induces a weak equivalence 
\begin{equation}\label{eqbase}
B\btLC X\to B\Omega^{n} \Sigma^{n}X_{+}\simeq 
\Omega^{n-1}\Sigma^{n}X_{+}.
\end{equation}
Since up to homotopy this map is compatible with the inclusion of
$\Sigma X_{+}$ in 
$B\btLC X$, this completes the argument in the case $n=1$.  For $n>1$,
we need 
the reduced free $\LC[j]$-space functor $C_{j}$ from
\cite[2.4]{MayGILS}; its fundamental formal property is that it gives an
adjunction between \emph{based} maps from a based spaced $Y$ to a
$\LC[j]$-space 
$Z$ and maps of $\LC[j]$-spaces from $C_{j}Y$ to $Z$.  Its fundamental
homotopical 
property is that when $Y$ is connected and nondegenerately based, the
universal map $C_{j}Y\to \Omega^{j}\Sigma^{j}Y$ is a weak
equivalence.  The fundamental formal property also holds when the
target $Z$ is a partial $\LC[j]$-space: For a based space $Y$, based
maps of partial power 
systems $Y\to Z$ are in bijective correspondence with maps of partial
$\LC[j]$-spaces $C_{j}Y\to Z$.

To be specific about the weak equivalence $B\Omega Z\simeq Z$
in~\eqref{eqbase}, we use the zigzag
in~\cite[14.3]{MayClass}  
\begin{equation}\label{eqmayclass}
B(\Omega Z)=B(\Omega_{M}Z) \xleftarrow{\,\simeq\,} 
B(P_{M}Z,\Omega_{M}Z,*)\xrightarrow{\,\simeq\,} Z,
\end{equation}
which is a weak equivalence whenever $Z$ is connected.
Here $P_{M}Z$ denotes the Moore based path space (the space of
positive length paths ending at the base point), the middle term
the classical two-sided bar construction for the action of the Moore
loop space on the Moore path space, and the maps are induced by the
trivial map $P_{M}Z\to *$ (on the left) and the start point projection
$P_{M}Z\to Z$ on the right.  When $Z$ is a $\LC[j]$-space,
we can make 
\[
B\subdot(P_{M}Z,\Omega_{M}Z,*)=
P_{M}Z\times \Omega_{M}Z\times \dotsb \times \Omega_{M}Z\times *
\]
a partial power system by taking the $m$-the partial power to be the
subset of the $m$-the power where the lengths match up, as for $B$
above.  Then both maps in the zigzag become maps of partial
$\LC[j]$-spaces (with the true power system for $Z$).

Returning to~\eqref{eqbase}, we have a zigzag of weak
equivalences of partial $\LC[n-1]$-spaces
\[
B\btLC X\to B\Omega^{n} \Sigma^{n}X_{+}
\from
B(P_{M}\Omega^{n-1}\Sigma^{n}X_{+},\Omega_{M}\Omega^{n-1}\Sigma^{n}X_{+},*)
\to \Omega^{n-1}\Sigma^{n}X_{+}.
\]
We can now see that the inclusion of $\Sigma X_{+}$ into $B\btLC X$
induces a weak equivalence of partial $\LC[n-1]$-spaces
\[
C_{n-1}\Sigma X_{+}\to B\btLC X
\]
(for the true $\LC[n-1]$-space $C_{n-1}\Sigma X_{+}$).
We use this as the base case of an inductive argument: Assume by
induction that the natural map $\Sigma^{i}X_{+}\to B^{i}\btLC X$
induces a weak equivalence of partial $\LC[n-i]$-spaces 
\[
C_{n-i}\Sigma^{i}X_{+}\to B^{i}\btLC X.
\]
Applying $B$, the weak equivalence of $\LC[n-i]$-spaces
$C_{n-i}\Sigma^{i}X_{+}\to \Omega^{n-i}\Sigma^{n}X_{+}$ and the
zigzag~\eqref{eqmayclass} give us a zigzag of weak equivalences of
partial $\LC[n-(i+1)]$-spaces
\begin{multline*}
\Omega^{n-(i+1)}\Sigma^{n}X_{+}\from
B(P_{M}\Omega^{n-(i+1)}\Sigma^{n}X_{+},\Omega_{M}\Omega^{n-(i+1)}\Sigma^{n}X_{+},*)\\
\to B\Omega^{n-i}\Sigma^{n}X_{+}\from
BC_{n-i}\Sigma^{i}X_{+}\to B^{i+1}\btLC X.
\end{multline*}
The inclusions of $\Sigma^{i+1}X_{+}$ into each of these spaces agree
up to homotopy under these maps, and so the induced map of partial
$\LC[n-(i+1)]$-spaces 
$C_{n-(i+1)}\Sigma^{i+1}X_{+}\to B^{i+1}\btLC X$ is a weak equivalence.
This completes the proof of Proposition~\ref{propdeloop}, which
completes the proof of Theorem~\ref{thmtq}.

\begin{rem}
For $Z=\Omega^{j}Y$ for a $(j-1)$-connected space $Y$, the identification
of~\eqref{eqmayclass} as a zigzag of weak equivalences of partial
$\LC[j]$-spaces implies by induction that $B^{n}$ is an
``$n$-fold de-looping machine''.   As an alternative argument, it
should be possible to deduce Proposition~\ref{propdeloop} from a
uniqueness theorem for $n$-fold de-looping machines such as
\cite{dunndeloop}; however, translating the problem to the context in
which such a theorem applies is more complicated than the direct
argument above.
\end{rem}

\section{Further Structure on the Bar Construction}\label{secopn}

With an eye to using Theorem~\ref{maintq} for computations, we take
this final section to verify two of the expected properties of the
multiplication on the bar construction.  We begin by studying the
diagonal map on the bar construction, and we show that it commutes
with the $\LC[n-1]$-multiplication constructed in
Section~\ref{secbar}.  We then study power operations,
showing that 
the (dimension shifting) map on homotopy groups from a non-unital
$\LC$-algebra $N$ to its bar construction $BN$
preserves power operations in the expected way.  In this section, we
work in the context of true algebras since that is where these remarks
are of primary interest.  

Given an augmented $R$-algebra $A$, it is well-known that the bar
construction $BA$ admits a diagonal map
\[
BA\to BA\sma_{R} BA
\]
that is associative up to homotopy, even up to coherent homotopy.  The
best construction of this map uses ``edgewise subdivision''
\cite[\S1]{BHM}.  For a simplicial object, $X\subdot$, the edgewise
subdivision is the object $\sd X\subdot$ where $\sd X_{n}=X_{2n+1}$.
The argument for \cite[1.1]{BHM} shows that just as in the context of
simplicial sets or simplicial spaces, in the context of simplicial
$R$-modules, we have a natural isomorphism between the geometric
realization of $X\subdot$ and the geometric realization of the
edgewise subdivision $\sd X\subdot$.  We get the diagonal map on the
bar construction as the composite
\[
BA \iso |\sd B\subdot A| \to |B\subdot A\sma_{R} B\subdot A|\iso
BA\sma_{R}BA
\]
for a particular simplicial map $\sd B\subdot A\to B\subdot
A\sma_{R}B\subdot A$. 
This map in degree $m$ is the map
\[
(\sd B\subdot A)_{m} = A^{(2m+1)}\to A^{(m)}\sma_{R}A^{(m)}=
B_{m}A\sma_{R}B_{m}A
\]
that performs the augmentation $A\to R$ on the $(m+1)$-st factor of
$A$.  We prove the following theorem.

\begin{thm}\label{thmhopf}
Let $N$ be a non-unital $\LC$-algebra.  The diagonal map $BKN\to
BKN\sma_{R}BKN$ above is a
map of $\LC[n-1]$-algebras.
\end{thm}

\begin{proof}
The edgewise subdivision functor and isomorphism on geometric
realization preserve smash
products of $R$-modules in the sense that the diagram of natural
isomorphisms
\[
\xymatrix{%
|X\subdot \sma_{R} Y\subdot|\ar[rr]\ar[d]
&&|X\subdot|\sma_{R}|Y\subdot|\ar[d]\\
|\sd(X\subdot\sma_{R} Y\subdot)|\ar[r]
&|(\sd X\subdot)\sma_{R} (\sd Y\subdot)|\ar[r]
&|\sd X\subdot|\sma_{R}|\sd Y\subdot|
}
\]
commutes.  It therefore suffices to check that the map from
$\sd B\subdot A$ to $B\subdot A\sma_{R}B\subdot A$ is a map of
simplicial $\LC[n-1]$-algebras (for $A=MKN$), and this is clear from the
construction of the $\LC[n-1]$-structure.
\end{proof}

We close with a remark on power operations.
For technical reasons about homotopy groups,
we restrict to the context of $R$-modules of orthogonal spectra or
EKMM $S$-modules for this discussion.  Consider a non-unital true
$\LC$-algebra $N$ and choose a representative map of $R$-modules
$R^{q}_{c}\to N$ where $R^{q}_{c}$ is some cofibrant version of the
$q$-sphere $R$-module.  We then get a map of non-unital true
$\LC$-algebras $\btNC R^{q}_{c}\to N$, where (as above) $\btNC$ denotes the free
non-unital $\LC$-algebra functor in $R$-modules.  The induced map
\[
\bigoplus_{m>0} R_{*}(\LC(m)_{+}\sma_{\Sigma_{m}}S^{(mq)})
\iso \pi_{*}(\btNC R^{q}_{c})\to \pi_{*}N
\]
depends only on the original $x\in \pi_{q}N$ and not on the
choice of representative.  Restricting to the $m$-th homogeneous
piece, we get the map
\[
\oP^{m}(x)\colon 
R_{*}(\LC(m)_{+}\sma_{\Sigma_{m}}S^{(mq)})\to \pi_{*}N,
\]
which we think of as the total $m$-ary $\LC$-algebra power operation
of $x$; we think of $R_{*}(\LC(m)_{+}\sma_{\Sigma_{m}}S^{(mq)})$ as
parametrizing the $m$-ary power operations on $\pi_{q}N$.
We relate the power operations on $\pi_{q}$ to the power operations
on $\pi_{q+1}$ using the suspension sequence 
\[
\btNC R^{q}_{c}\to \btNC CR^{q}_{c}\to \btNC R^{q+1}_{c}.
\]
The composite map is the  trivial map and the middle term has a
canonical contraction; this then defines a map
\[
\pi_{*}(\btNC R^{q}_{c})\to \pi_{*+1}(\btNC R^{q+1}_{c})
\]
and in particular a map
\[
\sigma \colon R_{*}(\LC(m)_{+}\sma_{\Sigma_{m}}S^{(mq)})\to
R_{*+1}(\LC(m)_{+}\sma_{\Sigma_{m}}S^{(m(q+1))})
\]
In terms of our work above, we have the following result.

\begin{thm}\label{thmopn}
For a non-unital true $\LC$-algebra $N$, the canonical map 
\[
\sigma \colon \pi_{*}N\to \pi_{*+1}\tB N
\]
preserves $m$-ary $\LC[n-1]$-algebra power
operations for all $m$, meaning that the diagrams
\[
\xymatrix{%
R_{*}(\LC[n-1](m)_{+}\sma_{\Sigma_{m}}S^{(mq)})
\ar[r]^{\sigma}\ar[d]_{\oP^{m}(x)}&
R_{*+1}(\LC[n-1](m)_{+}\sma_{\Sigma_{m}}S^{(m(q+1))})
\ar[d]^{\oP^{m}(\sigma x)}\\
\pi_{*}N\ar[r]_{\sigma}&\pi_{*+1}\tB N
}
\]
commute for all $x\in \pi_{q}N$.
Here we regard $N$ as a non-unital $\LC[n-1]$-algebra via the last
coordinates embedding of $\LC[n-1]$ in $\LC[n]$.
\end{thm}

\begin{proof}
We have an ``$E$'' version of the bar construction where $E_{0}N=K\MN$
and 
\[
E\subdot N = 
\underbrace{K\MN \sma_{R} \dotsb \sma_{R}K\MN}%
_{\ssdot\ \text{factors}} \sma_{R} K\MN
\]
for $\ssdot >0$.  Likewise, we have a $\tE$ version such that $EN=K\tE
N$.  Constructions analogous to those above make these 
into partial non-unital $\LC[n-1]$-algebras, and the inclusion of $N$
in $\MN$ in $\tE_{0}N$ (as, say, $\{1\}_{+}\sma N\subset P_{+}\sma N$)
induces a map of partial $\LC[n-1]$-algebras $N\to \tE N$.  The trivial map
$\MN\to *$ induces a map of partial $\LC[n-1]$-algebras $\tE N\to \tB N$,
giving us a sequence of maps of partial $\LC[n-1]$-algebras
\[
N \to \tE N \to \tB N
\]
with the composite map $N\to \tB N$ the trivial map.  The usual
simplicial contraction argument shows that $\tE$ is contractible.
Choosing a contraction, any map of $R$-modules $R^{q}_{c}\to N$ gives
us a map of partial power systems $CR^{q}_{c}\to \tE N$ and hence a
map of partial power systems $\Sigma R^{q}_{c}\to \tB N$.  This
correspondence lifts the map $\pi_{*}N\to \pi_{*+1}\tB N$ to
representatives.  Applying the free functor $\bNC[n-1]$ and unwinding
the definition of 
\[
\sigma\colon R_{*}(\LC[n-1](m)_{+}\sma_{\Sigma_{m}}S^{(mq)})
\to R_{*+1}(\LC[n-1](m)_{+}\sma_{\Sigma_{m}}S^{(m(q+1))})
\]
above, the result follows.
\end{proof}



\bibliographystyle{plain}
\def\noopsort#1{}\def\MR#1{}

\end{document}